\pdfoutput=1
\documentclass[a4paper, 11pt, reqno]{amsart}
\calclayout
\usepackage{scrextend}
\usepackage{amssymb}
\usepackage{amsmath}
\usepackage{amsthm}
\usepackage[numbers]{natbib}
\usepackage{mathtools}
\usepackage{enumitem}
\usepackage{calc}
\usepackage{setspace}
\usepackage[printwatermark]{xwatermark}
\usepackage{xcolor}
\usepackage{graphicx}
\usepackage{tikz-cd}
\usepackage{mathscinet}
\usepackage{epigraph}
\usepackage{hyperref}
\usepackage[top=3 cm, bottom=3 cm, left=3 cm, right=3cm]{geometry}

\makeatletter
\def\ps@plain{\ps@empty
}
\makeatother
\usepackage{lipsum}
\newcommand\blfootnote[1]{%
  \let\thempfn\relax
  \footnotetext[0]{{#1}}
}
\newtheorem{thm}{Theorem}[section]

\newtheorem{conj}[thm]{Conjecture}
\newtheorem{lem}[thm]{Lemma}
\newtheorem{prop}[thm]{Proposition}

\numberwithin{equation}{section}

\DeclareMathOperator{\id}{id}

\DeclareMathOperator{\Gal}{\mathrm{Gal}}

\DeclareMathOperator{\End}{\mathrm{End}}

\DeclareMathOperator{\tors}{\mathrm{tors}}

\DeclareMathOperator{\Stab}{\mathrm{Stab}}

\DeclareMathOperator{\Pic}{\mathrm{Pic}}

\title{Torsion points on isogenous abelian varieties}
\author[G. A. Dill]{Gabriel A. Dill}
\address{University of Oxford, Mathematical Institute, Andrew Wiles Building, Radcliffe Observatory Quarter, Woodstock Road, Oxford OX2 6GG}
\email{gabriel.dill@maths.ox.ac.uk}
\date{\today}

\begin{document}

\begin{abstract}
Investigating a conjecture of Zannier, we study irreducible subvarieties of abelian schemes that dominate the base and contain a Zariski dense set of torsion points that lie on pairwise isogenous fibers. If everything is defined over the algebraic numbers and the abelian scheme has maximal variation, we prove that the geometric generic fiber of such a subvariety is a union of torsion cosets. We go on to prove fully or partially explicit versions of this result in fibered powers of the Legendre family of elliptic curves. Finally, we apply a recent result of Galateau-Mart\'{i}nez to obtain uniform bounds on the number of maximal torsion cosets in the Manin-Mumford problem across a given isogeny class. For the proofs, we adapt the strategy, due to Lang, Serre, Tate, and Hindry, of using Galois automorphisms that act on the torsion as homotheties to the family setting.
\end{abstract}


\maketitle

\tableofcontents

\section{Introduction}
The Manin-Mumford conjecture predicted that at most finitely many points on a curve of genus $g \geq 2$ become torsion when the curve is embedded in its Jacobian (in characteristic $0$). The conjecture was generalized to a statement about subvarieties of arbitrary dimension of a given abelian variety by Lang in \cite{LangBook} (see also \cite{MR0190146}) and was proven in this more general form in \cite{Raynaud} by Raynaud, who had already proven it in the case of curves in \cite{RaynaudCurve}. The analogous statement for linear tori, also conjectured to hold by Lang, was then proven by Laurent in \cite{MR767195}. Finally, Hindry proved the analogous statement for arbitrary commutative algebraic groups in \cite{MR969244}. In this article, subvarieties will always be closed. Varieties will be reduced, but not necessarily irreducible. Fields will always be of characteristic $0$.

More recently, Masser and Zannier have studied and proven a relative version of the Manin-Mumford conjecture for curves in families of abelian varieties over $\bar{\mathbb{Q}}$ in a series of papers culminating in \cite{MaZaPreprint}. Unlike the classical conjecture, this relative version, conjectured by Pink in \cite{PUnpubl}, is naturally concerned with unlikely intersections with positive-dimensional subvarieties rather than points. In a slightly different direction, Zannier proposed the following conjecture that again concerns unlikely intersections with points (Conjecture 1.4 in \cite{G15}):

\begin{conj}[Zannier]\label{conj:zannier}
Let $\pi: \mathfrak{A}_{g,l} \to A_{g,l}$ denote the universal family of complex principally polarized abelian varieties of dimension $g$ with symplectic level $l$-structure and let $A_0$ be a fixed complex abelian variety. Let $\mathcal{V} \subset \mathfrak{A}_{g,l}$ be an irreducible subvariety that contains a Zariski dense set of points $p \in \mathfrak{A}_{g,l}(\mathbb{C})$ such that the fiber $(\mathfrak{A}_{g,l})_{\pi(p)}$ is isogenous to $A_0$ and $p$ is torsion on $(\mathfrak{A}_{g,l})_{\pi(p)}$. Then $\pi(\mathcal{V})$ is a totally geodesic subvariety of $A_{g,l}$ and $\mathcal{V}$ is an irreducible component of a subgroup scheme of $\mathfrak{A}_{g,l} \times_{A_{g,l}} \pi(\mathcal{V}) \to \pi(\mathcal{V})$.
\end{conj}

Gao has proven in \cite{G15} that Conjecture \ref{conj:zannier} holds if $\dim \pi(\mathcal{V}) \leq 1$. Orr had previously shown in \cite{MR3377393} that any curve in the moduli space of complex principally polarized abelian varieties of dimension $g$ that contains infinitely many points corresponding to pairwise isogenous abelian varieties is totally geodesic. The use of isogenies also allows to formulate a relative version of the Mordell-Lang conjecture, known as the Andr\'e-Pink-Zannier conjecture (Conjecture 1.2 in \cite{G15}). It is a consequence of Pink's more general Conjecture 1.6 in \cite{MR2166087} on intersections of subvarieties of mixed Shimura varieties with generalized Hecke orbits. Special cases of (variants of) the Andr\'e-Pink-Zannier conjecture and Conjecture \ref{conj:zannier} have been proven by Habegger in \cite{MR3181568}, by Pila in \cite{MR3164515}, by Lin and Wang in \cite{MR3383643}, by Gao in \cite{G15}, and by the author in \cite{D19} and \cite{D18}.

Most of these results concern families of abelian varieties whose base variety is a curve and most of them are proven via the Pila-Zannier strategy of o-minimal point counting that goes back to Pila and Zannier's new proof of the Manin-Mumford conjecture in \cite{MR2411018}. Consequently, the effectivity of these results is sometimes unclear. However, Binyamini's results in \cite{Binyamini} and \cite{Bin20} suggest that at least the o-minimal point counting could in principle be made effective. For the Pila-Zannier strategy and unlikely intersections in general, see Zannier's book \cite{MR2918151}.

The purpose of this article is to apply another classical method to this problem: We adapt the use of the Galois operation on the torsion points of an abelian variety, due to Lang, Serre, Tate, and Hindry (see \cite{MR0190146} and \cite{MR969244}), to the family setting. The fundamental observation that makes this approach work is the following: If $A_0$ is an abelian variety over a number field $K$ with fixed algebraic closure $\bar{K}$ and $\sigma \in \Gal(\bar{K}/K)$ acts on the torsion of $(A_0)_{\bar{K}}$ as a homothety, then $\sigma$ fixes every finite subgroup of $(A_0)_{\bar{K}}$. Hence, for any quotient $A$ of $(A_0)_{\bar{K}}$ by a finite subgroup, i.e. for any abelian variety isogenous to $(A_0)_{\bar{K}}$, the conjugate of $A$ by $\sigma$ is isomorphic to $A$. The existence of enough such $\sigma$ acting on the torsion of $(A_0)_{\bar{K}}$ as homotheties is guaranteed by a theorem of Serre (\cite{MR1730973}, No. 136, Th\'{e}or\`{e}me 2').

Applying this approach in a family setting seems to be new and yields both qualitatively and quantitatively new results.

Qualitatively, we essentially prove the ``vertical" half of the conclusion in Zannier's conjecture in Section \ref{sec:mmisog} in the case where everything is defined over $\bar{\mathbb{Q}}$, allowing the base variety to be of arbitrary dimension:

\begin{thm}\label{thm:isogenymaninmumford}
Let $S$ be an irreducible variety, defined over $\bar{\mathbb{Q}}$. Fix an algebraic closure $\overline{\bar{\mathbb{Q}}(S)}$ of $\bar{\mathbb{Q}}(S)$ and let $\xi$ denote the geometric generic point of $S$ with residue field $\overline{\bar{\mathbb{Q}}(S)}$. Let $\pi: \mathcal{A} \to S$ be a principally polarized abelian scheme of relative dimension $g$ over $S$, also defined over $\bar{\mathbb{Q}}$. Let $\eta$ denote the generic point of $S$ and suppose that the natural morphism $\rho: S \to A_g$ to the coarse moduli space $A_g$ of principally polarized abelian varieties of dimension $g$ over $\bar{\mathbb{Q}}$ satisfies $|\rho^{-1}(\rho(\eta))| < \infty$.

Let $\mathcal{V} \subset \mathcal{A}$ be an irreducible subvariety such that $\pi(\mathcal{V}) = S$. Fix an abelian variety $A_0$, defined over $\bar{\mathbb{Q}}$. Suppose that the set of $x \in \mathcal{V}(\bar{\mathbb{Q}})$ such that $x$ is a torsion point of the fiber $\mathcal{A}_{\pi(x)}$ and $\mathcal{A}_{\pi(x)}$ is isogenous to $A_0$ is Zariski dense in $\mathcal{V}$. Then $\mathcal{V}_\xi$ is equal to a union of translates of abelian subvarieties of $\mathcal{A}_\xi$ by torsion points of $\mathcal{A}_\xi$.
\end{thm}

Note however that the conclusion in Theorem \ref{thm:isogenymaninmumford} concerns irreducible components of algebraic subgroups of the geometric generic fiber instead of irreducible components of subgroup schemes and an abelian subvariety of the generic fiber is not always the generic fiber of an abelian subscheme (see Lemma 2.9 and the following counterexample in \cite{BD}). Nevertheless, one can use ``spreading out" to show that the conclusion in Theorem \ref{thm:isogenymaninmumford} is optimal if $\mathcal{A}$ contains a Zariski dense set of fibers that are isogenous to $A_0$.

The condition on the morphism $\rho$ is satisfied for example if $\rho$ is quasi-finite. For Theorem \ref{thm:isogenymaninmumford} to hold, some condition on $\rho$ is clearly necessary. For example, if $\mathcal{A} \to S$ is an abelian scheme of positive relative dimension with a Zariski dense set of pairwise isogenous fibers, then the image of the diagonal section of the abelian scheme $\mathcal{A} \times_{S} \mathcal{A} \to \mathcal{A}$ contains a Zariski dense set of torsion points that lie on pairwise isogenous fibers. However, the geometric generic fiber of the diagonal section is not a torsion point.

Quantitatively, we can apply the method to prove fully or partially explicit results. For this, we turn to a concrete example in Section \ref{sec:legendrecurve}: Let $Y(2) = \mathbb{A}_{\mathbb{Q}}^{1}\backslash\{0,1\}$ and let $\mathcal{E} \hookrightarrow Y(2) \times_{\mathbb{Q}} \mathbb{P}_{\mathbb{Q}}^2 \subset \mathbb{P}_{\mathbb{Q}}^1  \times_{\mathbb{Q}} \mathbb{P}_{\mathbb{Q}}^2$ be the Legendre family of elliptic curves over $Y(2)$, defined by the equation $Y^2Z = X(X-Z)(X-\lambda Z)$, where $\lambda$ is the affine coordinate on $Y(2)$ and $[X:Y:Z]$ are homogeneous projective coordinates on $\mathbb{P}_{\mathbb{Q}}^2$. Both $Y(2)$ and $\mathcal{E}$ are varieties over $\mathbb{Q}$. For $g \in \mathbb{N} = \{1,2,\hdots\}$, we denote the $g$-fold fibered power $\mathcal{E} \times_{Y(2)} \cdots \times_{Y(2)} \mathcal{E}$ by $\mathcal{E}^{(g)}$. The structural morphism is again denoted by $\pi: \mathcal{E}^{(g)} \to Y(2)$.

If everything is defined over $\bar{\mathbb{Q}}$, then the purely qualitative statement of ``Manin-Mumford with isogenies" is known in this case by \cite{MR3181568}. If furthermore the subvariety is a curve or the fixed abelian variety is a power of an elliptic curve without CM, then the qualitative statement of ``Mordell-Lang with isogenies" is also known by \cite{D19} and \cite{D18}. The new features of the results we present here are their full or partial explicitness and sometimes their effectivity.

We say that a multihomogeneous polynomial of multidegree $(d_1,\hdots,d_k)$ has multidegree at most $(D_1,\hdots,D_k)$ if $d_i \leq D_i$ ($i=1,\hdots,k$). The following theorem for curves in the Legendre family is completely explicit:

\begin{thm} \label{thm:effectiveisogenymm}
Let $K$ be a number field with a fixed algebraic closure $\bar{K}$. Let $\mathcal{C} \subset \mathcal{E}_K$ be a (maybe reducible) curve such that each of its irreducible components surjects onto $Y(2)_K$. Suppose that $\mathcal{C}$ is defined in $\mathcal{E}_K \subset \mathbb{P}_K^1 \times_K \mathbb{P}_K^2$ by bihomogeneous polynomials of bidegree at most $(D_1,D_2) \in \mathbb{N}^2$ with coefficients in $K$. Let $E_0$ be an elliptic curve, defined over $K$, let $j(E_0) \in K$ denote its $j$-invariant, and let $h(E_0)$ denote its stable Faltings height.

Suppose that $p \in \mathcal{C}(\bar{K})$ is torsion on $\mathcal{E}_{\pi(p)}$ and $\mathcal{E}_{\pi(p)}$ is isogenous to $(E_0)_{\bar{K}}$. Then the order of $p$ is bounded by
\[ \max\left\{(3CD_2)^{4},\exp\left(2^{\frac{18}{5}}\right)\right\},\]
where
\[ C = \begin{cases}
\exp\left(1.9 \cdot 10^{10}\right)\left([K:\mathbb{Q}]\max\{1,h(E_0),\log [K:\mathbb{Q}]\}\right)^{12395} & \text{(if $E_0$ has no CM),}\\
6[K:\mathbb{Q}(j(E_0))] & \text{(if $E_0$ has CM)}.\\
\end{cases}\]
\end{thm}

The two cases correspond to $E_0$ having CM or not, i.e. the endomorphism ring of $(E_0)_{\bar{K}}$ being larger than or isomorphic to $\mathbb{Z}$. The stable Faltings height that we use is normalized so that it is equal to the height $h_F$ in \cite{MR3225452}. However, the choice of normalization is relevant only for Theorems \ref{thm:effectiveisogenymm} and \ref{thm:effectiveisogenyml} and their proofs. By an easy modification of our proof, the exponent of $3CD_2$ in the upper bound can be improved to any $\kappa > 1$ at the expense of worsening the exponent $\frac{18}{5}$. The good quality of the upper bound in the CM case allows us to recover and make explicit (over $\bar{\mathbb{Q}}$) a result of Andr\'e from Lecture IV in \cite{Andre} in the case of the Legendre family: any non-torsion (multi)section of the Legendre family takes at most finitely many torsion values at CM arguments.

In Section \ref{sec:legendrecurve}, we also obtain a result of Mordell-Lang type. In order to formulate it, we define the height of a polynomial with algebraic coefficients as the (absolute logarithmic) height of the vector of its coefficients, seen as a point in projective space. See Definition 1.5.4 in \cite{MR2216774} for a definition of the height on projective space.

\begin{thm}\label{thm:effectiveisogenyml}
Let $\bar{\mathbb{Q}}$ denote a fixed algebraic closure of $\mathbb{Q}$. Let $\mathcal{C} \subset \mathcal{E}$ be a (maybe reducible) curve such that each of its irreducible components surjects onto $Y(2)$. Suppose that $\mathcal{C}$ is defined in $\mathcal{E} \subset \mathbb{P}_{\mathbb{Q}}^1 \times_{\mathbb{Q}} \mathbb{P}_{\mathbb{Q}}^2$ by bihomogeneous polynomials with coefficients in $\mathbb{Q}$ of bidegree at most $(D_1,D_2) \in \mathbb{N}^2$ and height at most $\mathcal{H}$. Let $E_0$ be an elliptic curve, defined over $\mathbb{Q}$, and let $h(E_0)$ denote its stable Faltings height.

Set $\gamma_1 = 12698$, $\gamma_2 = 2.2\cdot 10^{10}$, and $\gamma_3 = 26471$. Suppose that $p \in \mathcal{C}(\bar{\mathbb{Q}})$ satisfies $p = \varphi(q)$ for some isogeny $\varphi: (E_0)_{\bar{\mathbb{Q}}} \to \mathcal{E}_{\pi(p)}$ with cyclic kernel and a non-torsion point $q \in E_0(\bar{\mathbb{Q}})$ in the divisible hull of $E_0(\mathbb{Q})$. Then
\[ \deg \varphi \leq \max\{2,h(E_0)\}^{\gamma_1}\max\{D_1,D_2,\mathcal{H}\}^{6}\]
and there exists a natural number $N$ such that $Nq \in E_0(\mathbb{Q})$ and
\[ N \leq \exp(\gamma_2)\max\{1,h(E_0)\}^{\gamma_3}\max\{D_1,D_2,\mathcal{H}\}^{9}.\]
\end{thm}

Here we have to assume that the base field is $\mathbb{Q}$. All of our results for the Legendre family hinge on an effective version of Serre's Open Image Theorem for elliptic curves without CM, due to Lombardo in \cite{MR3437765}, as well as on an effective version of the analogous result for elliptic curves with CM, due to Bourdon and Clark in \cite{BourdonClark} (see Theorem \ref{thm:lombardo} below). As remarked by Bourdon and Clark, the latter result is a consequence of earlier work of Stevenhagen \cite{Stevenhagen}. Weaker results had been obtained earlier in the CM case by Lombardo in \cite{MR3766118} and by Eckstein in \cite{Eckstein}.

The proof of Theorem \ref{thm:effectiveisogenyml} follows the strategy used in \cite{MR3383643} of obtaining an upper and a lower bound for the height of a certain point such that the two bounds are incompatible for $\deg \varphi$ large enough. In \cite{MR3383643}, this strategy was applied to isogeny orbits of finitely generated groups. In order to be able to apply the same strategy to isogeny orbits of (certain) groups of finite rank, we use an explicit and uniform Kummer theoretic result of Lombardo and Tronto in \cite{LT21}, which is the reason for the restriction on the base field.

In Section \ref{sec:legendrevar}, we turn to higher-dimensional subvarieties of fibered powers of the Legendre family and obtain the following result. See Section \ref{sec:prelim} for our conventions concerning degrees. Unless explicitly stated otherwise, degrees of subvarieties of (base changes of) $\mathcal{E}^{(g)}$ are always taken with respect to (base changes of) its natural immersion in $\mathbb{P}^1_{\mathbb{Q}} \times_{\mathbb{Q}} \left(\mathbb{P}^2_{\mathbb{Q}}\right)^g$. 

\begin{thm}\label{thm:effectiveisogenymmposdim}
Let $K$ be a number field with a fixed algebraic closure $\bar{K}$. Let $g \in \mathbb{N}$ and let $\mathcal{V} \subset \mathcal{E}^{(g)}_K$ be a subvariety. Let $E_0$ be an elliptic curve, defined over $K$, and let $h(E_0)$ denote its stable Faltings height. Fix an algebraic closure $\overline{K(Y(2))}$ of $\bar{K}(Y(2))$ and let $\xi$ denote the geometric generic point of $Y(2)_{\bar{K}}$ with residue field $\overline{K(Y(2))}$.

Suppose that $p \in \mathcal{V}(\bar{K})$ is torsion on $\mathcal{E}^g_{\pi(p)}$ and $\mathcal{E}_{\pi(p)}$ is isogenous to $(E_0)_{\bar{K}}$. There exists an effective constant $\gamma(g)$, depending only on $g$, such that one of the following holds:
\begin{enumerate}
\item There exist a torsion point $q \in \mathcal{E}^g_{\xi}$ and an abelian subvariety $B$ of $\mathcal{E}^g_{\xi}$ such that $p \in \overline{q+B}(\bar{K})$ and $\overline{q+B} \subset \mathcal{V}$, where $\overline{q+B}$ denotes the Zariski closure in $\mathcal{E}^{(g)}_K$ of the image of $q+B$ under the natural morphism $\mathcal{E}^g_{\xi} \to \mathcal{E}^{(g)}_K$. The order of $q$ is bounded by $\max\{2,\deg \mathcal{V}\}^{\gamma(g)}$ and $\deg B \leq \max\{2,\deg \mathcal{V}\}^{\gamma(g)}$.
\item There exists an isogeny $\varphi: (E_0)_{\bar{K}} \to \mathcal{E}_{\pi(p)}$ with
\[ \deg \varphi \leq 2^{\max\{2,h(E_0),[K:\mathbb{Q}]\}^{\gamma(g)}}\max\left\{2,h(E_0),\deg \mathcal{V}\right\}^{\gamma(g)}.\]	
If $E_0$ has CM, the dependency on $h(E_0)$ can be omitted in the exponent.
\end{enumerate}
\end{thm}

Note that $q$ and $B$ in case (1) in Theorem \ref{thm:effectiveisogenymmposdim} are controlled by $g$ and $\deg \mathcal{V}$ alone; there is no dependency on the field of definition $K$. Recall that a subvariety of an abelian variety is called a torsion coset if it is a translate of an abelian subvariety by a torsion point. In order to bound the order of $q$ in terms of only $g$ and $\deg \mathcal{V}$, we apply an upper bound for the number of maximal torsion cosets contained in $\mathcal{V}_\xi$, due to David and Philippon in \cite{DavidPhilippon2007}, together with lower bounds for the degree of a torsion point of $\mathcal{E}^g_{\xi}$ over $\bar{K}(Y(2))$ in terms of its order.

On the other hand, the bound for the degree of the isogeny in case (2) must clearly involve $[K:\mathbb{Q}]$. It is not clear whether the dependency on $h(E_0)$ here and in the bound for the order of $p$ in Theorem \ref{thm:effectiveisogenymm} is also necessary; Coleman's conjecture with an upper bound that is polynomial in the degree of the number field would yield a bound that is independent of $h(E_0)$ (see Proposition 2.13 in \cite{MR3815154}, Section 2 of \cite{MR3437765}, and Th\'eor\`eme 1.2 in \cite{R17}). In order to bound the degree of the isogeny, we use a result of Gaudron-R\'{e}mond in \cite{MR3225452}, which improves and makes explicit earlier results of Masser-W\"ustholz in \cite{MR1037140} and \cite{MR1217345}. The above-mentioned results of Lombardo in \cite{MR3437765} and Bourdon-Clark in \cite{BourdonClark} are again essential for the proof in case (2).

Let us call a fiber $\mathcal{V}_{\pi(p)}$ for a point $p$ that does not fall under case (1) an exceptional fiber. One can then try to bound, independently of the field of definition of $\mathcal{V}$, the number of exceptional fibers as well as the number of maximal torsion cosets in each exceptional fiber. Note that the degree of such a maximal coset can be bounded in terms of only $g$ and $\deg \mathcal{V}$ thanks to Theorem 1 in \cite{MR587337}.

Combining Pila's results in \cite{MR3164515} with automatic uniformity in the form of Theorem 2.4 in \cite{MR2097105} (see also Corollary 3.5.9 in \cite{MR1854232}) seems to yield such bounds that depend only on $g$, $(E_0)_{\bar{K}}$, and $\deg \mathcal{V}$ (in an unspecified way). Using the homothety approach, we have not been able to prove such a bound for the number of exceptional fibers. We can however establish bounds for the number of maximal torsion cosets in each exceptional fiber that depend on the field of definition of $E_0$ and its stable Faltings height, but are independent of the field of definition of $\mathcal{V}$. For this, we combine our observations in Section \ref{sec:galmar} with the work of Galateau-Mart\'inez \cite{GT17} to make their result uniform across isogeny classes. We obtain the following theorem:

\begin{thm} \label{thm:countingacrossisogenyclass}
Let $K$ be a number field with a fixed algebraic closure $\bar{K}$ and let $A_0$ denote an abelian variety of dimension $g$ defined over $K$. There exists a constant $C = C(A_0,K)$ such that the following holds:

Let $A$ be an abelian variety, defined over $\bar{K}$, that is isogenous to $(A_0)_{\bar{K}}$. Suppose that $A$ is embedded in some $\mathbb{P}_{\bar{K}}^N$ as a projectively normal subvariety by means of the third tensor power of a symmetric ample line bundle. Let $V \subset A$ be a positive-dimensional subvariety. Let $\delta(V)$ denote the smallest natural number $d$ such that $V$ is the intersection of hypersurfaces of $A$ which have degree at most $d$. For $j = 0,\hdots,\dim V$, let $V^j_{\tors}$ denote the union of all $j$-dimensional components of the Zariski closure of the set of torsion points of $A$ that lie on $V$. Then
\[ \deg(V^j_{\tors}) \leq CN^{(g-j)\dim V}\deg(A)\delta(V)^{g-j}.\]
\end{thm}

After the completion of this article, the uniform Manin-Mumford conjecture was proved for curves by K\"uhne in \cite{Kuehne21} and in general by Gao, Ge, and K\"uhne in \cite{GGK}, both building on earlier work by Dimitrov, Gao, and Habegger \cite{DGH}; however, these results are not explicit. See \cite{DKY} for another recent uniform Manin-Mumford result, where the subvariety is a curve, but the abelian variety is not restricted to a fixed isogeny class. If $A_0 = E_0^g$ for some elliptic curve $E_0$ over $K$, then it follows from our proof of Theorem \ref{thm:countingacrossisogenyclass} together with the results in \cite{MR3437765} and \cite{BourdonClark} (see Theorem \ref{thm:lombardo} below) that $C(A_0,K)$ in Theorem \ref{thm:countingacrossisogenyclass} can be bounded effectively in terms of only $g$, $[K:\mathbb{Q}]$, and $h(E_0)$ and that the dependency on $h(E_0)$ can be omitted if $E_0$ has CM. Again, Coleman's conjecture with an upper bound that is polynomial in the degree of the number field would yield a bound that is independent of $h(E_0)$ also if $E_0$ does not have CM (see Proposition 2.13 in \cite{MR3815154} and Section 2 of \cite{MR3437765}). The recent work \cite{GM20} by Galateau and Mart\'{i}nez surveys what is known in general about the constant $C(A_0,K)$.

\section{Preliminaries}\label{sec:prelim}

Euler's phi function and the function counting the number of prime divisors will be denoted by $\phi$ and $\omega$ respectively. We will use the profinite integers $\hat{\mathbb{Z}} = \varprojlim \mathbb{Z}/n\mathbb{Z}$.

If $A$ is an abelian variety, then $0_A$ denotes its neutral element, $A_{\tors}$ denotes the set of its torsion points (over a fixed algebraic closure of its field of definition), $A[N]$ denotes the set of elements of $A_{\tors}$ of order dividing $N \in \mathbb{N}$, and $\hat{A}$ denotes the dual abelian variety. Let $\iota: A \to A$ denote the inversion morphism; a line bundle $\mathcal{L}$ on $A$ is called symmetric if $\iota^{\ast}\mathcal{L} \simeq \mathcal{L}$ and antisymmetric if $\iota^{\ast}\mathcal{L} \simeq \mathcal{L}^{\otimes(-1)}$. For each $N \in \mathbb{N}$, we can multiply elements of $A[N]$ by elements of $\mathbb{Z}/N\mathbb{Z}$ and thus we can multiply elements of $A_{\tors}$ and $\varprojlim A[N]$ by elements of $\hat{\mathbb{Z}}$. If $V \subset A$ is a subvariety, then $\Stab(V,A)$ denotes the stabilizer of $V$ in $A$; it is an algebraic subgroup of $A$. The $j$-invariant of an elliptic curve $E$ will be denoted by $j(E)$.

We use the (absolute logarithmic) height $h: \mathbb{P}^n(\bar{\mathbb{Q}}) \to [0,\infty)$ ($n \in \mathbb{N}$) as defined in Definition 1.5.4 in \cite{MR2216774}. Via the Segre embedding, this induces a height on any multiprojective space and on any open subset of a multiprojective variety over $\bar{\mathbb{Q}}$. The height however depends on the (multi-)projective embedding. The height $h(P)$ of a polynomial $P$ with algebraic coefficients is the height of the vector of its coefficients, seen as a point in projective space. We refer to Section 1.5 of \cite{MR2216774} for fundamental properties of the height.

If $A$ is an abelian variety over $\bar{\mathbb{Q}}$, embedded in some projective space through use of a symmetric line bundle, one can define a canonical (logarithmic) height $\widehat{h}_A: A(\bar{\mathbb{Q}}) \to [0,\infty)$ associated to the embedding. In particular, this applies if $A$ is an elliptic curve embedded in $\mathbb{P}^2_{\bar{\mathbb{Q}}}$ by means of a Weierstrass model. For the definition and properties of the canonical height, we refer to Sections 9.2 and 9.3 of \cite{MR2216774}.

Let $K$ be a field (as always of characteristic $0$). The Zariski closure of a subset $\Sigma$ of a variety $V$ over $K$ is denoted by $\overline{\Sigma}$. We use the following general convention: Let $n_1$, \dots, $n_k \in \mathbb{N}$. If $U$ is an open Zariski dense subset of a subvariety $V = \overline{U}$ of the multiprojective space $\mathbb{P}^{n_1}_K \times_K \cdots \times_K \mathbb{P}^{n_k}_K$, then $\deg U$ denotes the degree of the image of $V$ under the Segre embedding
\[ \mathbb{P}^{n_1}_K \times_K \cdots \times_K \mathbb{P}^{n_k}_K \hookrightarrow \mathbb{P}^{(n_1+1)\cdots(n_k+1)-1}_K.\]
If several immersions of $U$ as an open Zariski dense subset of a multiprojective variety are in play, we will always specify with respect to which one we take the degree. The degree of an arbitrary subvariety is the sum of the degrees of its irreducible components; consequently, the degree of the empty set is defined to be $0$. Furthermore, the degree of a subvariety, defined over $K$, is equal to the degree of its base change to any algebraic closure of $K$. The degree of a subvariety of an abelian variety with respect to some projective embedding is invariant under translation by rational points of the abelian variety as the corresponding cycles are algebraically equivalent.

In the following, we record some results about degrees of multiprojective varieties that we are going to use several times in this article.

\begin{thm}[B\'ezout]\label{thm:bezout}
Let $K$ be a field, let $n_1,\hdots, n_k \in \mathbb{N}$, let $U \subset \mathbb{P}^{n_1}_K \times_K \cdots \times_K \mathbb{P}^{n_k}_K$ be an open subset, and let $V,W$ be subvarieties of $U$. Then $\deg(V \cap W) \leq (\deg V)(\deg W)$.
\end{thm}

\begin{proof}
By our definition of the degree, we have $\deg(V \cap W) = \deg(\overline{V \cap W})$, $\deg V = \deg \overline{V}$, and $\deg W = \deg \overline{W}$. As $V = \overline{V} \cap U$ and $W = \overline{W} \cap U$, we find that $V \cap W = \overline{V} \cap \overline{W} \cap U$. It follows that the irreducible components of $\overline{V \cap W}$ are precisely the irreducible components of $\overline{V} \cap \overline{W}$ that have non-empty intersection with $U$. We deduce that $\deg(\overline{V \cap W}) \leq \deg(\overline{V} \cap \overline{W})$, so it suffices to prove the theorem for $U = \mathbb{P}^{n_1}_K \times_K \cdots \times_K \mathbb{P}^{n_k}_K$, $V = \overline{V}$, and $W = \overline{W}$.

In the case where $k = 1$ and $V$ and $W$ are irreducible, this then follows from Example 8.4.6 in \cite{MR732620}. Using the Segre embedding, we directly deduce the general case.
\end{proof}

\begin{thm}[B\'ezout, second version]\label{thm:bezoutii}
Let $K$ be a field, let $d, n_1,\hdots, n_k \in \mathbb{N}$, and let $V \subset \mathbb{P}^{n_1}_K \times_K \cdots \times_K \mathbb{P}^{n_k}_K$ be a subvariety. Let $W$ be an irreducible component of the intersection of $V$ with the common zero locus of a finite set of multihomogeneous polynomials of multidegree at most $(d,d,\hdots,d)$. Then
\[ \deg W \leq (\deg V)d^{\dim V-\dim W}. \]
\end{thm}

\begin{proof}
We identify subvarieties of $\mathbb{P}^{n_1}_K \times_K \cdots \times_K \mathbb{P}^{n_k}_K$ with their images under the Segre embedding and iterate the following step: We first choose an irreducible component $V_0$ of $V$ that contains $W$. If $V_0 = W$, we stop. If $V_0 \neq W$, our hypothesis on $W$ implies that we can find a hypersurface of degree at most $d$ in $\mathbb{P}^{\prod_{i=1}^{k}{(n_i+1)}-1}_K$ that contains $W$, but not $V_0$. We then replace $V$ by the intersection of $V_0$ with such a hypersurface. After at most $\dim V - \dim W$ steps, we end up with a variety that contains $W$ as an irreducible component. The theorem then follows from Theorem \ref{thm:bezout}.
\end{proof}

\begin{thm}\label{thm:faltings}
Let $K$ be a field, let $n_1,\hdots, n_k \in \mathbb{N}$, and let $V \subset \mathbb{P}^{n_1}_K \times_K \cdots \times_K \mathbb{P}^{n_k}_K$ be a subvariety. Then $V$ is defined in $\mathbb{P}^{n_1}_K \times_K \cdots \times_K \mathbb{P}^{n_k}_K$ by multihomogeneous polynomials of multidegree at most $(\deg V,\deg V,\hdots,\deg V)$.
\end{thm}

\begin{proof}
In the case where $k = 1$ and $V$ is equidimensional, this follows from Proposition 2.1 in \cite{MR1109353}. Using the Segre embedding, we then directly deduce the general case.
\end{proof}

\begin{lem}\label{lem:degproj}
Let $K$ be a field, let $n_1,\hdots,n_k \in \mathbb{N}$, and let $U$ be an open Zariski dense subset of a subvariety $V \subset \mathbb{P}^{n_1}_K \times_K \cdots \times_K \mathbb{P}^{n_k}_K$. Let $\pi$ be a projection to some collection of factors of the product. Then $\deg \overline{\pi(U)} \leq \deg U$.
\end{lem}

\begin{proof}
First of all, we have $\overline{\pi(U)} = \pi(V)$ since $\pi$ is closed. Since $\deg U = \deg V$ by our definition, we can assume without loss of generality that $U = V$. We can also assume that $K$ is algebraically closed and $V$ is irreducible.

We use the following fact: For a linear projection $p: \mathbb{P}^n_K\backslash L \to \mathbb{P}^{n-l}_K$ with center $L \subset \mathbb{P}^n_K$ of dimension $l-1$ and an irreducible subvariety $X \subset \mathbb{P}^n_K$, let $Y$ denote the Zariski closure of $p((\mathbb{P}^n_K\backslash L) \cap X)$. Then $\deg Y \leq \deg X$. For a proof of this fact, see \cite{MR1854232}, p. 55. Note that there $\bar{L}$ should be chosen such that $\bar{L} \cap \bar{V} \subset \theta(V\backslash C)$.

Let $S \subset \{1,\hdots,k\}$ be such that $\pi: \prod_{i=1}^{k}{\mathbb{P}^{n_i}_K} \to \prod_{i \in S}{\mathbb{P}^{n_i}_K}$. The lemma now follows from the above fact together with the following diagram:
\begin{equation*}
\begin{tikzcd}
\prod_{i=1}^{k}{\mathbb{P}^{n_i}_K} \arrow[r,hook] \arrow[d] & \mathbb{P}^{(n_1+1)\cdots(n_k+1)-1}_K \arrow[d,dashed] \\
\prod_{i \in S}{\mathbb{P}^{n_i}_K} \arrow[r,hook] & \mathbb{P}^{\prod_{i \in S}{(n_i+1)}-1}_K,
\end{tikzcd}
\end{equation*}
where the horizontal arrows are Segre embeddings and the right vertical arrow is a suitable linear coordinate projection, chosen such that the image of $V$ under the Segre embedding is not contained in its center.
\end{proof}

\section{Manin-Mumford with isogenies}\label{sec:mmisog}
 
In this section, we prove Theorem \ref{thm:isogenymaninmumford}. We recall the setting and set up some notation: We have a principally polarized abelian scheme $\pi: \mathcal{A} \to S$ of relative dimension $g$ over an irreducible variety $S$ over $\bar{\mathbb{Q}}$. In order to prove the theorem, we may replace $S$ by an open Zariski dense subset. Hence, we can and will assume without loss of generality that $S$ is affine. We have fixed an algebraic closure $\overline{\bar{\mathbb{Q}}(S)}$ of $\bar{\mathbb{Q}}(S)$ and $\xi$ is the geometric generic point of $S$ with residue field $\overline{\bar{\mathbb{Q}}(S)}$. We denote the zero section of $\mathcal{A}$ by $\epsilon$. We are also given a fixed abelian variety $A_0$ over $\bar{\mathbb{Q}}$.

We fix a number field $K$ over which $S$, $\mathcal{A}$ (together with its polarization), and $A_0$ are defined. In this section, we identify all varieties over $K$ with their base changes to $\bar{\mathbb{Q}}$ and ``irreducible" will always mean ``geometrically irreducible" when the base field is contained in $\bar{\mathbb{Q}}$ (unless explicitly specified otherwise). In particular, a homomorphism between two abelian varieties, both defined over some number field, is not assumed to be defined over the ground field and two abelian varieties, both defined over some number field, are called isogenous if they are isogenous over $\bar{\mathbb{Q}}$.

The natural morphism $\rho: S \to A_g$ to the coarse moduli space of principally polarized abelian varieties of dimension $g$, which is defined over $K$, satisfies $|\rho^{-1}(\rho(\eta))| < \infty$, where $\eta$ is the generic point of $S$. After maybe replacing $S$ by an open Zariski dense subset, we can and will assume without loss of generality that $\rho$ is quasi-finite with fibers of cardinality at most $M_1 \in \mathbb{N}$.

Let $\lambda$ denote the principal polarization on $\mathcal{A} \to S$, let $\hat{\mathcal{A}}$ denote the dual abelian scheme of $\mathcal{A}$, and let $\mathcal{P}$ denote the Poincar\'{e} line bundle on $\mathcal{A} \times_S \hat{\mathcal{A}}$. By Proposition 6.10 in \cite{MR1304906}, the polarization $2\lambda$ is induced by the line bundle $\mathcal{L} = (\id_{\mathcal{A}},\lambda)^{\ast}\mathcal{P}$ on $\mathcal{A}$. In Proposition 6.10 in \cite{MR1304906}, the abelian scheme $\mathcal{A}$ is assumed to be projective over $S$; this assumption is however unnecessary as it is only used to ensure that the dual abelian scheme $\hat{\mathcal{A}}$ exists, which is guaranteed by Theorem 1.9 in Chapter I of \cite{MR1083353}.

The restriction of $\mathcal{L}$ to each fiber of $\mathcal{A} \to S$ is symmetric by Theorem 8.8.4 in \cite{MR2216774}. The restrictions are also ample as $2\lambda$ is a polarization and ampleness of a line bundle on an abelian variety is preserved under algebraic equivalence. By Th\'eor\`{e}me 4.7.1 in \cite{MR217085} and Proposition 13.63 in \cite{MR2675155}, the line bundle $\mathcal{L}$ is relatively ample for $\pi$ as defined in Definition 13.60 in \cite{MR2675155}.

Since $S$ is affine, the line bundle $\mathcal{L}$ is ample. Thanks to Theorem II.7.6 in \cite{MR0463157}, there exists an immersion $\mathcal{A} \hookrightarrow \mathbb{P}_{\bar{\mathbb{Q}}}^{R_2} \times_{\bar{\mathbb{Q}}} S$ associated to the $l$-th tensor power of $\mathcal{L}$, all defined over $K$, for some $l \in \mathbb{N}$ large enough and some $R_2 \in \mathbb{N}$. As this immersion is proper, it is actually a closed embedding by \cite{stacks-project}, Tag 01IQ. In particular, $\mathcal{A}$ is projective over $S$. For $s \in S$, we denote by $\mathcal{A}_s$ the fiber of $\mathcal{A}$ over $s$ and by $\mathcal{L}_s$ the restriction of $\mathcal{L}$ to $\mathcal{A}_s$.

Since $S$ is affine, there is a closed embedding $S \hookrightarrow \mathbb{A}_{\bar{\mathbb{Q}}}^{R_1}$, defined over $K$, for some $R_1 \in \mathbb{N}$. By composing with the open immersion $\mathbb{A}_{\bar{\mathbb{Q}}}^{R_1} \hookrightarrow \mathbb{P}_{\bar{\mathbb{Q}}}^{R_1}$ and the Segre embedding, we obtain an immersion $\mathcal{A} \hookrightarrow \mathbb{P}_{\bar{\mathbb{Q}}}^{R_1R_2+R_1+R_2}$, also defined over $K$. The degree $\deg$ of a subvariety of $\mathcal{A}$ is defined to be the projective degree of the Zariski closure of its image under this immersion.

We include the immersion and the morphism $\rho: S \to A_g$ in the data associated to $\mathcal{A}$ so that constants depending on $\mathcal{A}$ are also allowed to depend on the choice of immersion and on $M_1$.

Theorem \ref{thm:isogenymaninmumford} will follow from the following proposition together with Lemma \ref{lem:not-zariskidense-too}. The method of the proof of the proposition is the same as in the article \cite{MR969244} by Hindry. This method is based on work of Lang, Serre, and Tate \cite{MR0190146}.

\begin{prop}\label{prop:galoisoperation}
Let $\mathcal{V} \subset \mathcal{A}$ be a subvariety, defined over $K$. Suppose that $x \in \mathcal{V}(\bar{\mathbb{Q}})$ is a torsion point of the fiber $\mathcal{A}_{\pi(x)}$ and that $\mathcal{A}_{\pi(x)}$ is isogenous to $A_0$. Then one of the following two possibilities holds:
\begin{enumerate}
\item $x$ lies in a translate of a positive-dimensional abelian subvariety of $\mathcal{A}_{\pi(x)}$ that is contained in $\mathcal{V}_{\pi(x)}$, or
\item the order of $x$ is bounded by a constant that depends only on $A_0$, $K$, $\mathcal{A}$, and $\deg \mathcal{V}$.
\end{enumerate}
\end{prop}

We will use the following lemma to prove Proposition \ref{prop:galoisoperation}. It can be regarded as a uniform version within an isogeny class of a theorem of Serre (\cite{MR1730973}, No. 136, Th\'{e}or\`{e}me 2'). In the proof of this lemma, we crucially use that $\rho$ is quasi-finite.

\begin{lem}\label{lem:homothety}
There exists a constant $B \in \mathbb{N}$, depending only on $A_0$, $K$, and $\mathcal{A}$, such that for all $a, M \in \mathbb{N}$ with $\gcd(a,M) = 1$ there exists $\sigma \in \Gal(\bar{\mathbb{Q}}/K)$ with the following property: For all torsion points $x \in \mathcal{A}(\bar{\mathbb{Q}})$ of order $M$ such that $\mathcal{A}_{\pi(x)}$ is isogenous to $A_0$, we have $\sigma(\pi(x)) = \pi(x)$ and $\sigma(x) = a^Bx$.
\end{lem}

\begin{proof}
Let $a, M \in \mathbb{N}$ with $\gcd(a,M) = 1$ be given and fix $\hat{a} \in \hat{\mathbb{Z}}^{\ast}$ such that $\hat{a} \equiv a \mod M$. By Th\'eor\`eme 3 in \cite{MR1944805}, due to Serre (\cite{MR1730973}, No. 136, Th\'{e}or\`{e}me 2'), there exists a constant $c = c(A_0,K) \in \mathbb{N}$ such that there exists $\sigma_a \in \Gal(\bar{\mathbb{Q}}/K)$ acting on $\varprojlim A_0[n] \simeq \hat{\mathbb{Z}}^{2g}$ as multiplication by $\hat{a}^c$.

Let $x \in \mathcal{A}(\bar{\mathbb{Q}})$ be a torsion point of order $M$ such that the fiber $\mathcal{A}_{\pi(x)}$ is isogenous to $A_0$. Let $\varphi: A_0 \to \mathcal{A}_{\pi(x)}$ be an isogeny and let $y \in A_0(\bar{\mathbb{Q}})$ be a torsion point such that $\varphi(y) = x$. Choose $N \in \mathbb{N}$ large enough so that $y$ belongs to $A_0[N]$ and $\ker \varphi \subset A_0[N]$. The order $M$ of $x$ is equal to the greatest common divisor of all $n \in \mathbb{N}$ with $ny \in \ker \varphi$. In particular, $M$ divides $N$.

We now choose $\tilde{a} \in \mathbb{N}$ such that $\tilde{a} \equiv \hat{a} \mod N$ and replace $a$ by $\tilde{a}$. The Galois automorphism $\sigma$ will not depend on the choice of $\tilde{a}$, but only on $\sigma_a$ and $\mathcal{A}$. We deduce that $\sigma_a(y) = a^{c(A_0,K)}y$.

If $\lambda_s: \mathcal{A}_s \to \widehat{\mathcal{A}_s}$ denotes the principal polarization on $\mathcal{A}_s$ ($s \in S$), then we have that $\hat{\varphi} \circ \lambda_{\pi(x)} \circ \varphi$ is a polarization of $A_0$. By Th\'eor\`eme 1.2 in \cite{R17}, which improves and optimizes earlier results by Silverberg in \cite{MR1154704} and by Masser-W\"ustholz in \cite{MR1207211}, every homomorphism between $A_0$ and $\widehat{A_0}$ is defined over a Galois extension of $K$ of degree at most $F(g) = 4\beta(g)6^{2(g-1)}(g!)^2$, where $\beta(g) = 3$ if $g \not\in \{2,4,6\}$, while $\beta(2) = 8$, $\beta(4) = 75$, and $\beta(6) = \frac{49}{12}$. Hence, the polarization $\hat{\varphi} \circ \lambda_{\pi(x)} \circ \varphi$ of $A_0$ is fixed by $\sigma_a^{\circ F(g)!}$.

Since $\sigma_a$ acts on $A_0[N] \supset \ker \varphi$ as multiplication by a natural number that is coprime to $N$, we know that $\sigma_a^{\circ F(g)!}(\ker \varphi) = \ker \varphi$. Therefore, there is an isomorphism $\iota: \mathcal{A}_{\pi(x)} \to \mathcal{A}_{\sigma_a^{\circ F(g)!}(\pi(x))}$ such that $\sigma_a^{\circ F(g)!}(\varphi) = \iota \circ \varphi$ (see Theorem 5.13 in \cite{MR3729270}).

Since $\hat{\varphi} \circ \lambda_{\pi(x)} \circ \varphi$ is fixed by $\sigma_a^{\circ F(g)!}$ and the polarization $\lambda$ of $\mathcal{A}$ is defined over $K$, we deduce that
\[ \hat{\varphi} \circ \lambda_{\pi(x)} \circ \varphi = \widehat{\sigma_a^{\circ F(g)!}(\varphi)} \circ \lambda_{\sigma_a^{\circ F(g)!}(\pi(x))} \circ \sigma_a^{\circ F(g)!}(\varphi) = \hat{\varphi} \circ \hat{\iota} \circ \lambda_{\sigma_a^{\circ F(g)!}(\pi(x))} \circ \iota \circ \varphi.\]
Note that $\widehat{\sigma_a^{\circ F(g)!}(\varphi)} = \sigma_a^{\circ F(g)!}\left(\hat{\varphi}\right)$ since dualizing commutes with extending the base field. As $\varphi$ and $\hat{\varphi}$ are isogenies, it follows that $\lambda_{\pi(x)} = \hat{\iota} \circ \lambda_{\sigma_a^{\circ F(g)!}(\pi(x))} \circ \iota$, so $\mathcal{A}_{\pi(x)}$ and $\mathcal{A}_{\sigma_a^{\circ F(g)!}(\pi(x))}$ are isomorphic as polarized abelian varieties.

Hence, the point $\rho(\pi(x))$ is fixed by $\sigma_a^{\circ F(g)!}$. It follows that the finite set $\rho^{-1}(\rho(\pi(x)))$ of cardinality at most $M_1$ is permuted by $\sigma_a^{\circ F(g)!}$ and therefore the Galois automorphism $\sigma_a^{\circ F(g)!M_1!}$ fixes $\pi(x)$.

By Th\'eor\`eme 1.2 in \cite{R17}, the isogeny $\varphi$ is defined over a Galois extension of $K(\pi(x))$ of degree at most $F(g)$ with $F(g)$ as above. We deduce that $\sigma = \sigma_a^{\circ M_1!(F(g)!)^2}$ fixes $\varphi$ and
\[ a^Bx = \varphi(a^By) = \varphi\left(\sigma(y)\right) = \sigma(\varphi(y)) = \sigma(x)\]
with $B = c(A_0,K)M_1!(F(g)!)^2$.
\end{proof}

The other ingredient in the proof of Proposition \ref{prop:galoisoperation} is the following proposition, whose proof follows the proofs of Th\'{e}or\`{e}me 1 and Proposition 2 in \cite{MR969244}.

\begin{prop}\label{prop:hindry}
Let $F$ be a field with a fixed algebraic closure $\bar{F}$. Let $A$ be an abelian variety over $F$ of dimension $g$, embedded in some projective space through use of a symmetric very ample line bundle $L$. Suppose that there exists $c \in \mathbb{N}$ with the following property: For all $a, N \in \mathbb{N}$ with $\gcd(a,N) = 1$, there exists $\sigma_{a,N} \in \Gal(\bar{F}/F)$ such that $\sigma_{a,N}(q) = a^{c}q$ for all torsion points $q \in A(\bar{F})$ of order $N$.

There exists an effective constant $\gamma(g)$, depending only on $g$, such that the following holds: Let $V \subset A$ be a subvariety and let $p \in V(\bar{F})$ be a torsion point that is not contained in any translate of a positive-dimensional abelian subvariety of $A_{\bar{F}}$ that is contained in $V_{\bar{F}}$. Then the order of $p$ is bounded by
\[\max\left\{\exp\left(\gamma(g)c^2\right),(\deg V)^{\gamma(g)}\right\}.\]
\end{prop}

\begin{proof}
Let $p \in V(\bar{F})$ be a torsion point that is not contained in any translate of a positive-dimensional abelian subvariety of $A_{\bar{F}}$ that is contained in $V_{\bar{F}}$ and let $N$ denote the order of $p$. We want to show that $N$ is bounded from above as in the proposition.

Let $Y \subset V$ be an equidimensional subvariety such that $p \in Y(\bar{F})$ and $Y$ is a union of irreducible components of $V$. For $d \in \mathbb{N}$ fixed and $t \in \mathbb{Z}$, $t \geq 0$, set (as in \cite{MR969244})
\[ Y_t = \bigcap_{j=0}^{t}{\left[d^j\right]^{-1}(Y)},\]
where $[d^j]$ denotes the multiplication-by-$d^j$ morphism on $A$.

Suppose that $d = a^c$ for some $a \in \mathbb{N}$ that is coprime to $N$. Our hypothesis then implies that there exists $\sigma_{a,N} \in \Gal(\bar{F}/F)$ such that $\sigma_{a,N}(q) = dq$ for all torsion points $q \in A(\bar{F})$ of order $N$. Hence, we have that $d^j p = \sigma_{a,N}^{\circ j}(p) \in Y(\bar{F})$ for all $j \in \mathbb{Z}$, $j \geq 0$, and therefore $p \in Y_t(\bar{F})$ for all $t \in \mathbb{Z}$, $t \geq 0$.

Suppose furthermore that $\dim_p (Y_s)_{\bar{F}} = \cdots = \dim_p (Y_{s+k})_{\bar{F}} = m'$ for some integers $s \geq 1$, $k \geq 0$, and $m' \geq 1$. It follows that there exists an irreducible component $C$ of $(Y_s)_{\bar{F}}$ that contains $p$, has dimension $m'$, and is contained in $(Y_{s+k})_{\bar{F}}$. Hence, we have $C' = \left[d^{k}\right](C) \subset (Y_s)_{\bar{F}}$. Since $p' = d^{k} p = \sigma_{a,N}^{\circ k}(p)$ is a Galois conjugate of $p$ over $F$, we have that $\dim_{p'} (Y_s)_{\bar{F}} = \dim_{p} (Y_s)_{\bar{F}}$. Therefore $C'$ is an irreducible component of $(Y_s)_{\bar{F}}$. We deduce from Theorem \ref{thm:bezoutii} and Theorem \ref{thm:faltings} that
\begin{equation}\label{eq:degcbound}
\deg \left[d^{k}\right](C) = \deg C' \leq (\deg Y)\left(\max_{j=1,\hdots,s}{\deg \left[d^j\right]^{-1}(Y)}\right)^{\dim Y-m'}.
\end{equation}

Recall that $\Stab(C,A_{\bar{F}})$ denotes the stabilizer of $C$ in $A_{\bar{F}}$. As $p$ is not contained in a translate of a positive-dimensional abelian subvariety of $A_{\bar{F}}$ that is contained in $Y_{\bar{F}}$, we have $\dim \Stab(C,A_{\bar{F}}) = 0$. Since $L$ is symmetric, we have $[d]^{\ast}L \simeq L^{\otimes d^2}$ by Proposition 8.7.1 in \cite{MR2216774}. Hence, we can deduce from the projection formula that
\[ \deg \left[d^{j}\right](C) = d^{2jm'}\left|\Stab(C,A_{\bar{F}}) \cap \ker \left[d^{j}\right]\right|^{-1}(\deg C)\]
as well as
\begin{equation}\label{eq:degpreim}
\deg \left[d^j\right]^{-1}(Y) = (\deg Y)d^{2j(g-\dim Y)}
\end{equation}
for $j \in \mathbb{Z}$, $j \geq 0$.

Combining these two equations with \eqref{eq:degcbound}, we deduce that
\begin{equation}\label{eq:combinedequation}
d^{2km'}\left|\Stab(C,A_{\bar{F}})\right|^{-1}(\deg C) \leq (\deg Y)^{\dim Y-m'+1}d^{2s(g-\dim Y)(\dim Y-m')}.
\end{equation}

Since $C$ is an irreducible component of $\bigcap_{j=0}^{s}{\left(\left[d^j\right]^{-1}(Y)\right)_{\bar{F}}}$, Theorem \ref{thm:bezoutii} and Theorem \ref{thm:faltings} then imply together with \eqref{eq:degpreim} that
\begin{multline}\label{eq:degcboundtoo}
\deg C \leq (\deg Y)\left(\max_{j=1,\hdots,s}{\deg \left[d^j\right]^{-1}(Y)}\right)^{\dim Y-m'} \\
 = (\deg Y)^{\dim Y-m'+1}d^{2s(g-\dim Y)(\dim Y-m')}.
\end{multline}

Now $\Stab(C,A_{\bar{F}})$ is a union of irreducible components of
\[ (C-x_1) \cap \hdots \cap (C-x_{m'+1})\]
for well chosen $x_1, \hdots, x_{m'+1} \in C(\bar{F})$. Theorem \ref{thm:bezout} implies that
\[ |\Stab(C,A_{\bar{F}})| = \deg \Stab(C,A_{\bar{F}}) \leq (\deg C)^{m'+1}.\]

We deduce from this together with \eqref{eq:combinedequation} and \eqref{eq:degcboundtoo} that
\begin{align*}
d^{2km'} \leq (\deg C)^{m'}(\deg Y)^{\dim Y-m'+1}d^{2s(g-\dim Y)(\dim Y-m')} \\
\leq (\deg Y)^{(\dim Y-m'+1)(m'+1)}d^{2s(g-\dim Y)(\dim Y-m')(m'+1)}.
\end{align*}

We now assume that
\begin{equation}\label{eq:dlowerbound}
d^{2(g-\dim Y)} \geq \deg Y.
\end{equation}
Note that $g-\dim Y > 0$ as otherwise $p$ would be contained in a translate of a positive-dimensional abelian subvariety of $A_{\bar{F}}$ that is contained in $Y_{\bar{F}}$.

It then follows that
\[ k \leq \frac{(s+1)(\dim Y-m'+1)(m'+1)(g-\dim Y)}{m'} \leq \frac{(s+1)\Xi}{m'},\]
where
\[ \Xi = \left(\frac{(\dim Y+2)^2(g-\dim Y)}{4}\right).\]
Induction now shows that $\dim_p (Y_t)_{\bar{F}} = 0$ for some $t \leq t_0$, where $t_0$ is effective and depends only on $g$ and $\dim Y$.

The same holds for any conjugate of $p$ over $F$. By our hypothesis, there are at least $\phi(N)/\left(2c^{\omega(N)}\right)$ such conjugates. We therefore get a lower bound
\begin{equation}\label{eq:lowerboundyt}
\deg Y_t \geq \frac{\phi(N)}{2c^{\omega(N)}}.
\end{equation}

Applying \eqref{eq:degpreim} together with Theorem \ref{thm:bezout} and $t \leq t_0$ to bound the degree of $Y_t = \bigcap_{j=0}^{t}{\left[d^j\right]^{-1}(Y)}$, we obtain an upper bound
\begin{equation}\label{eq:upperboundyt}
\deg Y_t \leq (\deg Y)^{t_0+1}d^{(g-\dim Y)t_0(t_0+1)} \leq d^{(g-\dim Y)(t_0+1)(t_0+2)}.
\end{equation}

By Th\'eor\`eme 11 in \cite{MR736719}, we have $\omega(N) \leq \frac{7\log N}{5 \log \log N}$ if $N \geq 3$. We can also estimate
\[ \phi(N) = N\prod_{p' | N}{\frac{p'-1}{p'}} \geq N\prod_{j=2}^{\omega(N)+1}{\frac{j-1}{j}} = \frac{N}{\omega(N)+1} \geq \frac{N}{2\log N +1},\]
where the product runs over all primes $p' \in \mathbb{N}$ that divide $N$. For $N \geq 3$, it then follows from combining this with \eqref{eq:lowerboundyt} and \eqref{eq:upperboundyt} that
\[ \frac{N^{1-\frac{7\log c}{5 \log \log N}}}{2(2\log N+1)} \leq d^{(g-\dim Y)(t_0+1)(t_0+2)}.\]

If $N > \exp\left(c^{\frac{8}{5}}\right) \geq 2$, which we will assume from now on, this implies that
\[ \frac{N^{\frac{1}{8}}}{2(2\log N+1)} \leq d^{(g-\dim Y)(t_0+1)(t_0+2)}.\]
As the function $x \mapsto x^{\frac{1}{16}} \cdot (2(2\log x+1))^{-1}$ ($x \geq 1$) attains its minimum at $x = \exp(31/2)$, we deduce that
\[ \frac{\exp\left(\frac{31}{32}\right)}{64}N^{\frac{1}{16}} \leq d^{(g-\dim Y)(t_0+1)(t_0+2)}.\]

We want to obtain a contradiction for $N$ large enough, but we have to make sure that \eqref{eq:dlowerbound} is satisfied. Recall that $d = a^c$ for $a \in \mathbb{N}$ coprime to $N$ and set $b = c(g-\dim Y) > 0$. In order to obtain a contradiction, it suffices to find $a \in \mathbb{N}$, coprime to $N$, such that
\[ (\deg Y)^{\frac{(t_0+1)(t_0+2)}{2}} \leq a^{b(t_0+1)(t_0+2)} < \frac{N^{\frac{1}{16}}}{25}.\]
Simplifying the problem, we may look for a prime number $a$ that does not divide $N$ and satisfies
\[ (\deg Y)^{\frac{1}{2b}} \leq a < \left(\frac{N^{\frac{1}{16}}}{25}\right)^{\frac{1}{b(t_0+1)(t_0+2)}}.\]

By Corollary 1 of Theorem 2 in \cite{MR137689}, this is possible if $N \geq \left(25\cdot17^{b(t_0+1)(t_0+2)}\right)^{16}$, which we will assume from now on, and
\[ \omega(N) + (\deg Y)^{\frac{1}{2b}} < \left(\frac{N^{\frac{1}{16}}}{25}\right)^{\frac{1}{b(t_0+1)(t_0+2)}}\frac{b(t_0+1)(t_0+2)}{\frac{\log N}{16}-\log 25},\]
which is a consequence of the simpler inequality
\[ \frac{\log N}{16}\left(\omega(N) + (\deg Y)^{\frac{1}{2b}}\right) < \left(\frac{N^{\frac{1}{16}}}{25}\right)^{\frac{1}{b(t_0+1)(t_0+2)}}.\]
Thanks to the above bound for $\omega(N)$ and $N \geq 25^{16} \geq e^e$, this inequality in turn follows from
\[ (\log N)^2(\deg Y)^{\frac{1}{2b}} < \left(\frac{N^{\frac{1}{16}}}{25}\right)^{\frac{1}{b(t_0+1)(t_0+2)}}.\]

Recall that $Y$ is a union of irreducible components of $V$ and therefore $\deg Y \leq \deg V$. Furthermore, if $\epsilon \in (0,1)$ and $N \geq \exp\left(\frac{1}{\epsilon^2}\right) \geq \epsilon^{-\frac{1}{\epsilon}}$, then we have
\[ \log N \leq \frac{N^{\epsilon}}{\epsilon} \leq N^{2\epsilon}.\]
This implies that we get a contradiction as soon as $N \geq \exp((128b(t_0+1)(t_0+2))^2)$ and
\[ N^{\frac{1}{32b(t_0+1)(t_0+2)}}(\deg V)^{\frac{1}{2b}} < \left(\frac{N^{\frac{1}{16}}}{25}\right)^{\frac{1}{b(t_0+1)(t_0+2)}}.\]
This last inequality is equivalent to $N > 25^{32}(\deg V)^{16(t_0+1)(t_0+2)}$.

Since $\dim Y \in \{0,\hdots,g\}$, all this together implies that there exists an effective constant $\gamma(g)$, depending only on $g$, such that we get a contradiction if $N > \max\left\{\exp\left(\gamma(g)c^2\right),(\deg V)^{\gamma(g)}\right\}$. Thus, we have established the proposition.
\end{proof}

We can now prove Proposition \ref{prop:galoisoperation}.

\begin{proof}[Proof of Proposition \ref{prop:galoisoperation}]
Let $x \in \mathcal{V}(\bar{\mathbb{Q}})$ be a torsion point such that the fiber $\mathcal{A}_{\pi(x)}$ is isogenous to $A_0$. By Theorem \ref{thm:faltings}, the Zariski closure of $\mathcal{V}$ in $\mathbb{P}_{\bar{\mathbb{Q}}}^{R_1R_2+R_1+R_2}$ is cut out by homogeneous forms of degree at most $\deg \mathcal{V}$ with coefficients in $K$. By specialization, it follows that $\mathcal{V}_{\pi(x)}$ is cut out in $\mathbb{P}_{\bar{\mathbb{Q}}}^{R_2}$ by forms of degree at most $\deg \mathcal{V}$ with coefficients in $K(\pi(x))$. We can find $X \subset \mathcal{V}_{\pi(x)}$ equal to a union of irreducible components of $\mathcal{V}_{\pi(x)}$ such that $x \in X(\bar{\mathbb{Q}})$ and $X$ is defined over $K(\pi(x))$ and irreducible as a variety over $K(\pi(x))$. By Theorem \ref{thm:bezoutii} (over the field $K(\pi(x))$), the degree $\deg X$ of $X$ as a subvariety of $\mathbb{P}_{\bar{\mathbb{Q}}}^{R_2}$ is bounded from above by $(\deg \mathcal{V})^{R_2}$.

Suppose now that $x$ does not satisfy condition (1) in Proposition \ref{prop:galoisoperation}. Let $B \in \mathbb{N}$ be the constant provided by Lemma \ref{lem:homothety}. We apply Proposition \ref{prop:hindry} with $F = K(\pi(x))$, $\bar{F} = \bar{\mathbb{Q}}$, $A = \mathcal{A}_{\pi(x)}$, $c = B$, and $V = X$. It follows that $x$ satisfies condition (2) in Proposition \ref{prop:galoisoperation}.
\end{proof}

The next lemma will give us everything we need to prove Theorem \ref{thm:isogenymaninmumford}.

\begin{lem}\label{lem:not-zariskidense-too}
Let $\mathcal{V} \subset \mathcal{A}$ be an irreducible subvariety that dominates $S$. Suppose that all abelian subvarieties of $\mathcal{A}_\xi$ are defined over $\bar{\mathbb{Q}}(S)$ and that the stabilizer $\Stab(\mathcal{V}_\xi,\mathcal{A}_\xi)$ is finite. Then the union of all translates of positive-dimensional abelian subvarieties of $\mathcal{A}_s$ that are contained in $\mathcal{V}_s$ for some $s \in S(\bar{\mathbb{Q}})$ is contained in a proper subvariety of $\mathcal{V}$.
\end{lem}

The proof of Lemma \ref{lem:not-zariskidense-too} runs along similar lines as the proof of Lemma 3.4 in \cite{D19}. The difference is that the base variety $S$ is now allowed to have dimension bigger than $1$. Note that Lemma \ref{lem:not-zariskidense-too} could also be obtained as a consequence of the much more general Theorem 12.2 in \cite{G152}, at least for $\mathcal{A}$ contained in a suitable universal family and then for arbitrary $\mathcal{A}$ as well. However, we prefer to give a direct proof that does not make use of the language of mixed Shimura varieties.

\begin{proof}
We first perform a quasi-finite dominant base change $S' \to S$ such that $S'$ is smooth and irreducible and every irreducible component of $\mathcal{V}_\xi$ is defined over $\bar{\mathbb{Q}}(S') \subset \overline{\bar{\mathbb{Q}}(S)}$. Set $\mathcal{A}' = \mathcal{A} \times_{S} S'$. Let $\mathcal{V}'$ be an irreducible component of $\mathcal{V} \times_{S} S' \hookrightarrow \mathcal{A}'$ that dominates $S'$ (and hence dominates $\mathcal{V}$).

If $\zeta$ is a geometric generic point of $S'$, then $\Stab(\mathcal{V}'_{\zeta},\mathcal{A}'_{\zeta})$ must be finite. Otherwise it would contain a positive-dimensional abelian subvariety of $\mathcal{A}'_{\zeta}$, which we identify with $\mathcal{A}_{\xi}$, but as all abelian subvarieties of $\mathcal{A}'_{\zeta}$ are defined over $\bar{\mathbb{Q}}(S)$, this abelian subvariety would be contained in the stabilizer of $\mathcal{V}_\xi$, which could therefore not be finite. Furthermore, the generic fiber of $\mathcal{V}'$ is irreducible by Section 2.1.8 of Chapter 0 of \cite{MR0163908} and hence also $\mathcal{V}'_\zeta$ is irreducible by our choice of $S'$.

If the union of all translates of positive-dimensional abelian subvarieties of $\mathcal{A}_s$ that are contained in $\mathcal{V}_s$ for some $s \in S(\bar{\mathbb{Q}})$ is Zariski dense in $\mathcal{V}$, then the union of all translates of positive-dimensional abelian subvarieties of $\mathcal{A}'_{s'}$ that are contained in $\mathcal{V}'_{s'}$ for some $s' \in S'(\bar{\mathbb{Q}})$ is Zariski dense in $\mathcal{V}'$. So we can replace $\mathcal{A}$ and $\mathcal{V}$ by $\mathcal{A}'$ and $\mathcal{V}'$ and assume without loss of generality that $\mathcal{V}_\xi$ is irreducible.

Let $N \in \mathbb{N}$ be a natural number that is larger than the order of $\Stab(\mathcal{V}_\xi,\mathcal{A}_\xi)$. There are finitely many irreducible subvarieties $\mathcal{T}_1, \hdots, \mathcal{T}_R \subset \mathcal{A}$ such that each $\mathcal{T}_i$ dominates $S$ and the union of the $\mathcal{T}_i$ ($i=1,\hdots,R$) is equal to the set of torsion points of order $N$ on $\mathcal{A}$: First of all, every irreducible component of the pre-image of $\epsilon(S)$ under the multiplication-by-$N$ morphism $[N]$ dominates $S$ by Proposition 2.3.4(iii) in \cite{MR0199181} since $[N]$ is \'etale, so flat (see \cite{MR861974}, Proposition 20.7). Therefore, every irreducible component of $[N]^{-1}(\epsilon(S))$ is of dimension $\dim S$. The same holds for any $M \in \mathbb{N}$ that divides $N$. Furthermore, $[N]^{-1}(\epsilon(S))$ is smooth as $[N]$ is \'etale and $S$ is smooth. Hence, no two distinct irreducible components of $[N]^{-1}\left(\epsilon(S)\right)$ intersect each other. So every irreducible component of $[N]^{-1}\left(\epsilon(S)\right)$ is either contained in $\bigcup_{M|N, M \neq N}[M]^{-1}\left(\epsilon(S)\right)$ or disjoint from it and our claim follows.

We now consider $\mathcal{W}_i = \mathcal{V} \cap (\mathcal{V}+\mathcal{T}_i)$ for $i \in \{1,\hdots,R\}$. If this variety were equal to $\mathcal{V}$, then we would have $\mathcal{V} \subset \mathcal{V}+\mathcal{T}_i$ and so $\mathcal{V}_\xi \subset \mathcal{V}_\xi+(\mathcal{T}_i)_\xi$. For dimension reasons and thanks to the irreducibility of $\mathcal{V}_\xi$, we would get that $\mathcal{V}_\xi = t+\mathcal{V}_\xi$ for a torsion point $t \in \mathcal{A}_\xi\left(\overline{\bar{\mathbb{Q}}(S)}\right)$ of order $N$. This contradicts our choice of $N$. So $\mathcal{W}_i \subsetneq \mathcal{V}$.

On the other hand, each positive-dimensional abelian variety contains a point of order $N$, so the union of all translates of positive-dimensional abelian subvarieties of $\mathcal{A}_s$ that are contained in $\mathcal{V}_s$ for some $s \in S(\bar{\mathbb{Q}})$ is contained in $\bigcup_{i=1}^{R}{\mathcal{W}_i}$. As every $\mathcal{W}_i$ is a proper closed subset of $\mathcal{V}$ and $\mathcal{V}$ is irreducible, the lemma follows.
\end{proof}

We now prove Theorem \ref{thm:isogenymaninmumford}.

\begin{proof}[Proof of Theorem \ref{thm:isogenymaninmumford}]
Recall that we can and do assume without loss of generality that $S$ is affine and $\rho$ is quasi-finite. After a quasi-finite dominant base change $S' \to S$ with $S'$ affine and irreducible and after replacing $\mathcal{A}$ by $\mathcal{A} \times_{S} S'$ and $\mathcal{V}$ by an irreducible component of $\mathcal{V} \times_{S} S'$ that dominates $S'$ (and hence $\mathcal{V}$), we can furthermore assume that all abelian subvarieties of $\mathcal{A}_\xi$ are defined over $\bar{\mathbb{Q}}(S)$. Here and in the following, it might sometimes be necessary to replace the field of definition $K$ by a finite extension of $K$ and we will do this without further comments. Note that the principal polarization of $\mathcal{A}$ yields a principal polarization of $\mathcal{A} \times_{S} S'$, that the morphism $S' \to A_g$ factors through $S \to A_g$, and that we can construct a quasi-projective immersion of $\mathcal{A} \times_{S} S'$ with the same properties as the one of $\mathcal{A}$.

Let $A'$ be the irreducible component of $\Stab(\mathcal{V}_\xi,\mathcal{A}_\xi)$ that contains $0_{\mathcal{A}_\xi}$. Then $A'$ is an abelian subvariety of $\mathcal{A}_\xi$. We can now use the Poincar\'{e} reducibility theorem to deduce that there exists another abelian subvariety $A''$ of $\mathcal{A}_\xi$ such that the natural morphism $A' \times_{\overline{\bar{\mathbb{Q}}(S)}} A'' \to \mathcal{A}_\xi$ given by restricting the addition morphism $\mathcal{A}_\xi \times_{\overline{\bar{\mathbb{Q}}(S)}} \mathcal{A}_\xi \to \mathcal{A}_\xi$ is an isogeny. Note that this morphism as well as $A'$ and $A''$ are defined over $\bar{\mathbb{Q}}(S)$.

By ``spreading out" (see Theorem 3.2.1 and Table 1 on pp. 306--307 in \cite{MR3729254}), we can find abelian schemes $\mathcal{A}'$ and $\mathcal{A}''$ over an open Zariski dense subset $U$ of $S$ with geometric generic fibers $A'$ and $A''$ and a morphism $\alpha: \mathcal{A}' \times_U \mathcal{A}'' \to \mathcal{A} \times_S U$ that extends the isogeny $A' \times_{\overline{\bar{\mathbb{Q}}(S)}} A'' \to \mathcal{A}_\xi$. We can assume without loss of generality that $S = U$.

As $\alpha$ is dominant, proper, and maps the image of the zero section to the image of the zero section, it follows that $\alpha$ restricts to an isogeny on each fiber. It suffices to prove that the conclusion of the theorem holds for one of the irreducible components of $\alpha^{-1}(\mathcal{V})$ that dominate $\mathcal{V}$, which we call $\mathcal{V}'$, inside the family $\mathcal{A}' \times_{S} \mathcal{A}''$.

By construction, the geometric generic fiber $\mathcal{V}'_\xi$ is equal to $\mathcal{A}'_\xi \times_{\overline{\bar{\mathbb{Q}}(S)}} \mathcal{V}''_\xi$, where $\mathcal{V}''$ is the image of $\mathcal{V}'$ under the projection to $\mathcal{A}''$, and hence $\mathcal{V}' = \mathcal{A}' \times_{S} \mathcal{V}''$. Note that the projection morphism is proper, so $\mathcal{V}''$ is closed in $\mathcal{A}''$.

Let $\epsilon': S \to \mathcal{A}'$ denote the zero section of $\mathcal{A}'$ and set $\mathcal{V}''' = \alpha(\epsilon'(S) \times_{S} \mathcal{V}'') \subset \mathcal{A}$. By construction, the stabilizer $\Stab(\mathcal{V}'''_\xi,\mathcal{A}_\xi)$ is finite and the set of torsion points $x \in \mathcal{V}'''(\bar{\mathbb{Q}})$ such that $\mathcal{A}_{\pi(x)}$ is isogenous to $A_0$ is Zariski dense in $\mathcal{V}'''$.

Combining Proposition \ref{prop:galoisoperation} with Lemma \ref{lem:not-zariskidense-too} shows that $\mathcal{V}'''$ must be equal to an irreducible component of $[M]^{-1}\left(\epsilon(S)\right)$ for some $M \in \mathbb{N}$, where $[M]: \mathcal{A} \to \mathcal{A}$ denotes the multiplication-by-$M$ morphism. The theorem follows.
\end{proof}

\section{Explicit results in the Legendre family -- curves}\label{sec:legendrecurve}

Recall that $Y(2) = \mathbb{A}_{\mathbb{Q}}^{1}\backslash\{0,1\}$ and that $\mathcal{E} \hookrightarrow Y(2) \times_{\mathbb{Q}} \mathbb{P}_{\mathbb{Q}}^2 \subset \mathbb{P}_{\mathbb{Q}}^1 \times_{\mathbb{Q}} \mathbb{P}_{\mathbb{Q}}^2$ is the Legendre family of elliptic curves over $Y(2)$ defined by the equation $Y^2Z = X(X-Z)(X-\lambda Z)$, where $\lambda$ is the affine coordinate on $Y(2)$ and $[X:Y:Z]$ are homogeneous projective coordinates on $\mathbb{P}_{\mathbb{Q}}^2$. Both $Y(2)$ and $\mathcal{E}$ are varieties over $\mathbb{Q}$. There is a natural surjective morphism $\pi: \mathcal{E} \to Y(2)$. The $j$-invariant defines a morphism $j: Y(2) \to \mathbb{A}_{\mathbb{Q}}^1$.

In this section, we will prove fully explicit results on ``Manin-Mumford with isogenies" and ``Mordell-Lang with isogenies" for a curve in the Legendre family in the case where everything is defined over a number field (or over $\mathbb{Q}$) and every irreducible component of the curve dominates $Y(2)$. Our results also have implications for the case of a curve in an arbitrary fibered power of the Legendre family since we can project onto each factor of the fibered power. We first prove the following useful lemma, a more explicit and precise version of Lemma \ref{lem:homothety}:

\begin{lem}\label{lem:uniformhomothety}
Let $K$ be a number field with a fixed algebraic closure $\bar{K}$. Let $E_0$ be an elliptic curve, defined over $K$. Let $c = c(E_0/K)$ be the ``Serre constant", i.e. the smallest natural number $c$ such that for any $a \in \hat{\mathbb{Z}}^{\ast}$, there is a $\tau_a \in \Gal(\bar{K}/K)$ that acts on the torsion of $(E_0)_{\bar{K}}$ as multiplication by $a^c$. Then for any $a \in \hat{\mathbb{Z}}^{\ast}$, there is a $\sigma_a \in \Gal(\bar{K}/K)$ such that for every $s \in Y(2)(\bar{K})$ such that $\mathcal{E}_{s}$ is isogenous to $(E_0)_{\bar{K}}$, $\sigma_a$ fixes $s$, fixes every isogeny from $(E_0)_{\bar{K}}$ to $\mathcal{E}_s$, and acts on the torsion of $\mathcal{E}_{s}$ as multiplication by $a^{2c}$. More precisely, if $\sigma \in \Gal(\bar{K}/K)$ acts on the torsion of $(E_0)_{\bar{K}}$ as multiplication by $b \in \hat{\mathbb{Z}}^{\ast}$, then for every $s \in Y(2)(\bar{K})$ such that $\mathcal{E}_{s}$ is isogenous to $(E_0)_{\bar{K}}$, $\sigma^{\circ 2}$ fixes $s$, fixes every isogeny from $(E_0)_{\bar{K}}$ to $\mathcal{E}_s$, and acts on the torsion of $\mathcal{E}_{s}$ as multiplication by $b^{2}$.
\end{lem}

\begin{proof}
Suppose that $\sigma \in \Gal(\bar{K}/K)$ acts as a homothety on the torsion of $(E_0)_{\bar{K}}$ and let $s \in Y(2)(\bar{K})$ be such that $\mathcal{E}_s$ is isogenous to $(E_0)_{\bar{K}}$. Let $\varphi: (E_0)_{\bar{K}} \to \mathcal{E}_{s}$ be any isogeny.

Since $\sigma$ acts as a homothety on the torsion of $(E_0)_{\bar{K}}$, it fixes $\ker \varphi$. By Theorem 5.13 in \cite{MR3729270}, there exists an isomorphism $\psi: \mathcal{E}_s \to \mathcal{E}_{\sigma(s)}$ such that $\sigma(\varphi) = \psi \circ \varphi$. Let $p_0 \in (\mathcal{E}_{s})_{\tors}$ be a point of order $2$ and let $q_0 \in (E_0)_{\tors}$ be a pre-image of $p_0$ under $\varphi$. Since $\sigma$ acts as a homothety on the torsion of $(E_0)_{\bar{K}}$, we have $\sigma(q_0) = b_0q_0$ for some odd $b_0 \in \mathbb{N}$. Since $p_0$ has order $2$, it follows that
\[ \psi(p_0) = \psi(b_0p_0) = (\psi \circ \varphi)(b_0q_0) = \sigma(\varphi)(\sigma(q_0)) = \sigma(p_0).\]

Because of the nature of the Legendre family, we deduce that $\sigma(s) = s$, which implies that $\psi$ is an automorphism of $\mathcal{E}_s$ that restricts to the identity on the $2$-torsion of $\mathcal{E}_s$. It follows that $\psi = \pm \id_{\mathcal{E}_s}$ and $\sigma^{\circ 2}(\varphi) = \varphi$.

Suppose now that $\sigma$ acts on the torsion of $(E_0)_{\bar{K}}$ as multiplication by $b \in \hat{\mathbb{Z}}^{\ast}$ and let $p \in (\mathcal{E}_{s})_{\tors}$. There exists $q \in (E_0)_{\tors}$ such that $\varphi(q) = p$. By the above, we know that $\sigma^{\circ 2}(s) = s$, $\sigma^{\circ 2}(\varphi) = \varphi$, and
\begin{align*}
\sigma^{\circ 2}(p) = \sigma^{\circ 2}(\varphi)\left(\sigma^{\circ 2}(q)\right) \\
= \varphi\left(b^{2}q\right) = b^{2}\varphi(q) = b^{2}p.
\end{align*}
The lemma follows.
\end{proof}

The following theorem is due to Lombardo in the non-CM case and Bourdon-Clark in the CM case:

\begin{thm}[Lombardo, Bourdon-Clark]\label{thm:lombardo}
Let $K$ be a number field with a fixed algebraic closure $\bar{K}$. Let $E_0$ be an elliptic curve, defined over $K$, let $j(E_0) \in K$ denote its $j$-invariant, and let $h(E_0)$ denote its stable Faltings height. The Serre constant $c = c(E_0/K)$ as defined in Lemma \ref{lem:uniformhomothety} satisfies $c \leq C$, where
\[ C = \begin{cases}
\exp\left(1.9 \cdot 10^{10}\right)\left([K:\mathbb{Q}]\max\{1,h(E_0),\log [K:\mathbb{Q}]\}\right)^{12395} & \text{($E_0$ has no CM),}\\
6[K:\mathbb{Q}(j(E_0))] & \text{($E_0$ has CM)}.\\
\end{cases}\]
More precisely, there is a subgroup $G \subset \hat{\mathbb{Z}}^{\ast}$ of index at most $C$ such that for every $a \in G$, there exists a $\sigma \in \Gal(\bar{K}/K)$ that acts on the torsion of $(E_0)_{\bar{K}}$ as multiplication by $a$.
\end{thm}

In the CM case, Eckstein had earlier obtained the weaker bound $48[K:\mathbb{Q}]$ for the Serre constant (Th\'eor\`eme 7 in \cite{Eckstein}), which Lombardo improved to $6[K:\mathbb{Q}]$ in Theorem 6.6 in \cite{MR3766118}. As Bourdon and Clark remark, their result is a consequence of earlier work of Stevenhagen \cite{Stevenhagen}.

\begin{proof}
In the non-CM case, the theorem follows from the improved version of Corollary 9.3 in \cite{MR3437765} that is mentioned in Remark 9.4 in \cite{MR3437765}. The result by Gaudron and R\'emond that is needed for this improvement and that was still unpublished when \cite{MR3437765} appeared is Corollaire 17.5 in \cite{GR20}.

In the CM case, let $L \subset \bar{K}$ denote the imaginary quadratic field such that $\left(\End (E_0)_{\bar{K}}\right) \otimes_{\mathbb{Z}} \mathbb{Q} \simeq L$. Since $[KL:L(j(E_0))] \leq [K:\mathbb{Q}(j(E_0))]$, the theorem follows from Corollary 1.5 in \cite{BourdonClark}.
\end{proof}

We can now prove Theorem \ref{thm:effectiveisogenymm}, an explicit instance of Theorem \ref{thm:isogenymaninmumford}.

\begin{proof}[Proof of Theorem \ref{thm:effectiveisogenymm}]
Let $p \in \mathcal{C}(\bar{K})$ be torsion on $\mathcal{E}_{\pi(p)}$ such that $\mathcal{E}_{\pi(p)}$ is isogenous to $(E_0)_{\bar{K}}$. Note that $K(\pi(p)) \subset K(p) \subset \bar{K}$. We will obtain a lower and an upper bound for the degree $[K(p):K(\pi(p))]$ that are incompatible with each other if the order of $p$ is sufficiently large.

For the upper bound, note that by our hypothesis $\{p\} \subset \mathcal{E}_{\pi(p)} \subset \{\pi(p)\} \times_{\bar{K}} \mathbb{P}_{\bar{K}}^{2}$ is an irreducible component of the intersection of the zero loci of $Y^2Z - X(X-Z)(X-\pi(p) Z)$ and some homogeneous polynomial of degree at most $D_2$ in $X,Y,Z$ with coefficients in $K(\pi(p))$. It follows from Theorem \ref{thm:bezout} that
\begin{equation}\label{eq:upperbound}
[K(p):K(\pi(p))] \leq 3D_2.
\end{equation}

For the lower bound, let $G \subset \hat{\mathbb{Z}}^{\ast}$ be the subgroup furnished by Theorem \ref{thm:lombardo}. By Theorem \ref{thm:lombardo}, the index of $G$ in $\hat{\mathbb{Z}}^{\ast}$ is less than or equal to $C$ for $C$ as in Theorem \ref{thm:effectiveisogenymm}. By Theorem \ref{thm:lombardo} together with Lemma \ref{lem:uniformhomothety}, for any $a \in G$, there is a $\sigma_a \in \Gal(\bar{K}/K)$ that fixes $\pi(p)$ and acts on the torsion of $\mathcal{E}_{\pi(p)}$ as multiplication by $a^{2}$.

Let $N$ be the order of $p$. We can suppose that $N \geq 3$. We obtain a first lower bound
\[[K(p):K(\pi(p))] \geq |Gp| \geq \frac{\phi(N)}{2^{\omega(N)}(2C)}.\]

By Th\'eor\`eme 11 in \cite{MR736719}, we have $\omega(N) \leq \frac{7\log N}{5 \log \log N}$. We can also estimate
\[ \phi(N) = N\prod_{p' | N}{\frac{p'-1}{p'}} \geq N\prod_{j=2}^{\omega(N)+1}{\frac{j-1}{j}} = \frac{N}{\omega(N)+1} \geq \frac{N}{2\log N +1},\]
where the product runs over all primes $p' \in \mathbb{N}$ that divide $N$. It follows that
\begin{equation}\label{eq:firstupperbound}
[K(p):K(\pi(p))] \geq \frac{N^{1-\frac{7\log 2}{5 \log \log N}}}{2C(2\log N +1)}.
\end{equation}

We now assume that $N \geq \exp\left(2^{\frac{18}{5}}\right) \geq \exp(12)$. The function $x \mapsto x^{\frac{13}{36}} \cdot (2(2\log x+1))^{-1}$ is monotonically increasing for $x \geq \exp(59/26)$. It follows that
\[ \frac{N^{\frac{13}{36}}}{2(2 \log N+1)} \geq \frac{\exp(13/3)}{50} > 1.\]
Together with $N \geq \exp\left(2^{\frac{18}{5}}\right)$ and \eqref{eq:firstupperbound}, this implies that
\[ [K(p):K(\pi(p))] \geq \frac{N^{\frac{11}{18}}}{2C(2 \log N +1)} > C^{-1}N^{\frac{1}{4}}.\]
Combining this bound with \eqref{eq:upperbound}, we deduce that
\[ N \leq (3CD_2)^{4}.\qedhere\]
\end{proof}

We next prove Theorem \ref{thm:effectiveisogenyml}.

\begin{proof}[Proof of Theorem \ref{thm:effectiveisogenyml}]
Let $p \in \mathcal{C}(\bar{\mathbb{Q}})$ such that $p = \varphi(q)$ for some isogeny $\varphi: (E_0)_{\bar{\mathbb{Q}}} \to \mathcal{E}_{\pi(p)}$ with cyclic kernel and a non-torsion point $q \in E_0(\bar{\mathbb{Q}})$ in the divisible hull of $E_0(\mathbb{Q})$. Let $N$ be the smallest natural number such that $Nq \in E_0(\mathbb{Q})$ and let $N_1$ be the smallest natural number such that $N_1q \in E_0(\mathbb{Q})+(E_0)_{\tors}$. Let $q_r \in E_0(\mathbb{Q})$ and $q_t \in (E_0)_{\tors}$ be such that $N_1q = q_r + q_t$. Let $N_2$ denote the order of $q_t$. Since $N$ divides $N_1N_2$, it suffices to bound $N_1$ as well as $\deg \varphi$ and $N_2$.

Set $E_0(\mathbb{Q})_{\tors} = E_0(\mathbb{Q}) \cap (E_0)_{\tors}$. Let $m$ be the largest natural number such that the class of $q_r$ in $E_0(\mathbb{Q})/E_0(\mathbb{Q})_{\tors}$ is divisible by $m$. Let $\tilde{q}_r \in E_0(\mathbb{Q})$ satisfy $m\tilde{q}_r - q_r \in E_0(\mathbb{Q})_{\tors}$. If $\gcd(m,N_1) > 1$, then $(N_1\gcd(m,N_1)^{-1})q \in E_0(\mathbb{Q})+(E_0)_{\tors}$, a contradiction. So $\gcd(m,N_1) = 1$.

We choose $\tilde{q} \in E_0(\bar{\mathbb{Q}})$ such that $N_1\tilde{q} = \tilde{q}_r$. It follows that $\tilde{q}_t := m\tilde{q} - q$ lies in $(E_0)_{\tors}$. By Theorem 6.5 in \cite{LT21}, an explicit version of Theorem 1.2 in \cite{LT19}, we have
\[ [\mathbb{Q}(\tilde{q},E_0[M_1]):\mathbb{Q}(E_0[M_1])] \geq 2^{-126}N_1^2\]
for every natural number $M_1$ that is divisible by $N_1$. We can choose $M_1$ so that it is divisible by $N_1N_2$ as well as by $2\deg \varphi$ and the order of $\tilde{q}_t$.

We then have $\mathbb{Q}(\tilde{q},E_0[M_1]) \supset \mathbb{Q}(q,E_0[M_1])$ since $m\tilde{q} = q + \tilde{q}_t$ and $\tilde{q}_t \in E_0[M_1]$. But since $m$ and $N_1$ are coprime and $N_1\tilde{q} = \tilde{q}_r \in E_0(\mathbb{Q})$, the two fields are equal. We deduce that the extension $\mathbb{Q}(q,E_0[M_1])/\mathbb{Q}(E_0[M_1])$ is Galois and that
\begin{equation}\label{eq:lowerkummerbund}
[\mathbb{Q}(q,E_0[M_1]):\mathbb{Q}(E_0[M_1])] \geq 2^{-126}N_1^2.
\end{equation}

There is an injective group homomorphism $\rho: \Gal(\mathbb{Q}(q,E_0[M_1])/\mathbb{Q}(E_0[M_1])) \hookrightarrow E_0[N] \simeq (\mathbb{Z}/N\mathbb{Z})^2$ given by $\sigma \mapsto \sigma(q)-q$. Therefore, the Galois group $\Gal(\mathbb{Q}(q,E_0[M_1])/\mathbb{Q}(E_0[M_1]))$ is isomorphic to a product of two cyclic groups. Its image under $\rho$ is even contained in $E_0[N_1]$ since $q_t \in E_0[M_1]$ and therefore
\[ N_1(\sigma(q)-q) = \sigma(N_1q) - N_1q = \sigma(q_r+q_t)-N_1q = q_r+q_t-N_1q = 0_{E_0}\]
for all $\sigma \in \Gal(\mathbb{Q}(q,E_0[M_1])/\mathbb{Q}(E_0[M_1]))$.

Any $\sigma \in \Gal(\bar{\mathbb{Q}}/\mathbb{Q}(E_0[M_1]))$ fixes the set $\left\{u \in E_0(\bar{\mathbb{Q}}); 2\varphi(u) = 0_{\mathcal{E}_{\pi(p)}}\right\}$ pointwise and hence $\tau = \sigma^{\circ 2}$ fixes $\pi(p)$ as well as $\varphi$ (cf. the proof of Lemma \ref{lem:uniformhomothety}). Furthermore, we have $\tau(p) = p$ if and only if $\varphi(\tau(q)) = \varphi(q)$ if and only if $\tau|_{\mathbb{Q}(q,E_0[M_1])}$ is mapped into $\ker \varphi \cap E_0[N_1]$ by $\rho$. Since $\ker \varphi$ is cyclic, we find that
\begin{align*}
[\mathbb{Q}(p):\mathbb{Q}(\pi(p))] \geq 2^{-2}2^{-126}|\ker \varphi \cap E_0[N_1]|^{-1}N_1^2 \\
\geq 2^{-128}N_1.
\end{align*}

On the other hand, we have
\begin{equation}\label{eq:boundfordegpoverpip}
[\mathbb{Q}(p):\mathbb{Q}(\pi(p))] \leq 3D_2
\end{equation}
as in the proof of Theorem \ref{thm:effectiveisogenymm}. This yields that
\begin{equation}\label{eq:boundfornone}
N_1 \leq 3 \cdot 2^{128} \cdot D_2 \leq 2^{130}D_2.
\end{equation}

We proceed with bounding $\deg \varphi$. Since each fiber $\mathcal{E}_s$ ($s \in Y(2)(\bar{\mathbb{Q}})$) is canonically embedded in $\mathbb{P}_{\bar{\mathbb{Q}}}^2$ by means of a Weierstrass model, we get an associated canonical height $\widehat{h}_{\mathcal{E}_{s}}: \mathcal{E}_s(\bar{\mathbb{Q}}) \to [0,\infty)$ on each fiber. We define a canonical height $\widehat{h}_{\mathcal{E}}: \mathcal{E}(\bar{\mathbb{Q}}) \to [0,\infty)$ by $\widehat{h}_{\mathcal{E}}(p') = \widehat{h}_{\mathcal{E}_{\pi(p')}}(p')$ for $p' \in \mathcal{E}(\bar{\mathbb{Q}})$.

We will bound $\deg \varphi$ by obtaining a lower and an upper bound for $\widehat{h}_{\mathcal{E}}(p)$ that are incompatible with each other for $\deg \varphi$ large enough. Let $\widehat{h}_{E_0}$ denote the canonical height on $E_0(\bar{\mathbb{Q}})$ induced by some Weierstrass model. The lower bound is easy: It follows from standard properties of the canonical height that
\begin{equation}\label{eq:heightlowerkummerbound}
\widehat{h}_{\mathcal{E}}(p) = (\deg \varphi)\widehat{h}_{E_0}(q) = N_1^{-2}(\deg \varphi)\widehat{h}_{E_0}(N_1q) = N_1^{-2}(\deg \varphi)\widehat{h}_{E_0}(q_r) \geq N_1^{-2}(\deg \varphi)h_0,
\end{equation}
where
\[ h_0 = \min_{q' \in E_0(\mathbb{Q})\backslash E_0(\mathbb{Q})_{\tors}}{\widehat{h}_{E_0}(q')} > 0.\]

For the upper bound, suppose that $p$ maps to $[x:y:z] \in \mathbb{P}^2(\bar{\mathbb{Q}})$ under the composition $\iota$ of the immersion $\mathcal{E} \hookrightarrow \mathbb{P}_{\mathbb{Q}}^1 \times_{\mathbb{Q}} \mathbb{P}_{\mathbb{Q}}^2$ with the projection to the second factor. Since $p$ is non-torsion, we have $z \neq 0$ and can assume without loss of generality that $z=1$. Set $\lambda = \pi(p) \in Y(2)(\bar{\mathbb{Q}})$ and $\iota_\lambda = \iota_{\bar{\mathbb{Q}}}|_{\mathcal{E}_\lambda}$.

It follows from the proof of Lemma 4.4 in \cite{MR3181568} that
\[ h(\iota_\lambda(2p')) \leq 4h(\iota_\lambda(p'))+ 3h(\lambda)+\log 72\]
for every $p' \in \mathcal{E}_{\lambda}(\bar{\mathbb{Q}})$, where $h$ denotes the height on $\mathbb{P}^2(\bar{\mathbb{Q}})$ and $\mathbb{P}^1(\bar{\mathbb{Q}}) \supset Y(2)(\bar{\mathbb{Q}})$. This implies that
\begin{equation}\label{eq:heightupperkummerbound}
\widehat{h}_{\mathcal{E}}(p) = \lim_{n \to \infty}{\frac{h(\iota_\lambda(2^np))}{4^n}} \leq h([x:y:1])+\frac{1}{3}(3h(\lambda)+\log 72) \leq h([x:y:1])+3\max\{1,h(\lambda)\}.
\end{equation}

Let $P((U,V),(X,Y,Z)) \in \mathbb{Q}[U,V,X,Y,Z]$ be one of the bihomogeneous polynomials of bidegree at most $(D_1,D_2)$ and height at most $\mathcal{H}$ defining $\mathcal{C}$. Set $Q(X,Y,Z) = P((\lambda,1),(X,Y,Z))$, a homogeneous polynomial of degree at most $D_2$. By hypothesis, we can choose $P$ such that $Q(X,Y,Z)$ and $F(X,Y,Z) = Y^2Z-X(X-Z)(X-\lambda Z)$ have at most finitely many common zeroes in $\mathbb{P}^2(\bar{\mathbb{Q}})$, among them $[x:y:1]$.

In \cite{MR1341770}, Philippon defines the (non-negative) height $h(V)$ of an irreducible subvariety $V$ of projective space over $\bar{\mathbb{Q}}$ (by Proposition 1.28 in \cite{MR3098424}, the height is independent from the choice of number field over which $V$ is defined). If $h_{\mathrm{R\acute{e}m}}$ denotes the height of a multihomogeneous form with coefficients in $\bar{\mathbb{Q}}$ as defined in \cite{MR1837829}, then it follows from the arithmetic B\'ezout theorem in the form of Th\'{e}or\`{e}me 3.4 and Corollaire 3.6 in \cite{MR1837829} that
\[ h(\{[x:y:1]\}) \leq  D_2\left(h_{\mathrm{R\acute{e}m}}(F)+\frac{3}{2}\right) + 3\left(h_{\mathrm{R\acute{e}m}}(Q)+\frac{D_2}{2}+\frac{D_2}{2}\log 2\right). \]
Since $h_{\mathrm{R\acute{e}m}}(G) \leq h(G) + \log\binom{\deg(G)+2}{2} + \frac{3}{4}\deg(G) \leq h(G) + \frac{9}{4}\deg(G)$ for $G \in \{F,Q\}$, where $\deg$ denotes the degree of a homogeneous form, we deduce that
\[ h(\{[x:y:1]\}) \leq 3h(Q) + D_2h(F) + 18D_2.\]

The height of $Q$ is bounded from above by $\log(D_1+1)+\mathcal{H}+D_1h(\lambda)$ while the height of $F$ is bounded from above by $2\max\{1,h(\lambda)\}$. Using that $h(p') \leq h(\{p'\})$ for any $\bar{\mathbb{Q}}$-point $p'$ of projective space by the computation at the bottom of p. 96 of \cite{MR1837829}, we deduce that
\begin{equation}\label{eq:boundforxyz}
h([x:y:1]) \leq 3h(Q) + D_2h(F) + 18D_2 \leq 20((D_1+D_2)\max\{1,h(\lambda)\}+\mathcal{H}).
\end{equation}

We proceed with bounding $h(\lambda)$. Set $j = j(\mathcal{E}_\lambda)$. From the equality
\[ \lambda^6 = \frac{1}{256}j\lambda^2(\lambda-1)^2+3\lambda^5-6\lambda^4+7\lambda^3-6\lambda^2+3\lambda-1\]
we can deduce the rather crude bound
\[ h(\lambda) \leq h(j) + \log 256 +\log 30.\]
It then follows from Theorem 1.1 in \cite{MR3925753} that
\begin{equation}\label{eq:boundforl}
h(\lambda) \leq h(j(E_0))+19+12\log \deg \varphi.
\end{equation}

Combining \eqref{eq:boundforxyz} and \eqref{eq:boundforl} yields
\begin{multline}\label{eq:boundforxyztoo}
h([x:y:1]) \leq 400((D_1+D_2)\max\{1,h(j(E_0))\}+\mathcal{H}+(D_1+D_2)\log \deg \varphi).
\end{multline}

Combining \eqref{eq:heightlowerkummerbound}, \eqref{eq:heightupperkummerbound}, \eqref{eq:boundforl}, and \eqref{eq:boundforxyztoo} and using $D_1 + D_2 \geq 2$ yields
\begin{multline}\label{eq:upperandlowerboundtogether}
N_1^{-2}(\deg \varphi)h_0 \leq 430((D_1+D_2)\max\{1,h(j(E_0))\}+\mathcal{H}+(D_1+D_2)\log \deg \varphi).
\end{multline}

Note that the inequality $A_3 \deg \varphi \leq A_1+A_2\log \deg \varphi$ (with $A_1,A_2,A_3 > 0$) implies that
\[ \deg \varphi \leq \max\left\{\frac{2A_1}{A_3},\frac{4A_2^2}{A_3^2}\right\}.\]
In order to show this, we may assume $A_3 = 1$ then if $\deg \varphi \geq 2A_1$ we have $\deg \varphi \leq 2A_2\log \deg \varphi \leq 2A_2(\deg \varphi)^{\frac{1}{e}}$ and hence $2A_2 \geq 1$ and $\deg \varphi \leq (2A_2)^{\frac{e}{e-1}} \leq 4A_2^2$.

We therefore obtain from \eqref{eq:upperandlowerboundtogether} that
\begin{equation}\label{eq:degphifirstbound}
 \deg \varphi \leq 2^{22}\max\{D_1,D_2,\mathcal{H}\}^{2}\max\{1,h(j(E_0))\}\max\{1,h_0^{-1}\}^2N_1^4.
 \end{equation}

It follows from Lemmas 2.6 and 3.2 in \cite{MR3925753} that
\[ h(j(E_0)) \leq 12h(E_0)+25+6\log(1+h(j(E_0))).\]
Since $12\log(1+u) \leq u$ for $u \in [50,\infty)$, this implies that
\[ \max\{1,h(j(E_0))\} \leq 74\max\{1,h(E_0)\}.\]
We deduce from this together with the upper bound for $h_0^{-1}$ in terms of $h(E_0)$ from Th\'eor\`eme 1.3 in \cite{MR3943753}, \eqref{eq:boundfornone}, and \eqref{eq:degphifirstbound} that
\begin{equation}\label{eq:boundfordegvarphi}
\deg \varphi \leq \max\{2,h(E_0)\}^{12698}\max\{D_1,D_2,\mathcal{H}\}^{6}.
\end{equation}

It remains to bound $N_2$. Recall that $N_1q = q_r + q_t$ with $q_r \in E_0(\mathbb{Q})$ and $q_t \in (E_0)_{\tors}$ of order $N_2$. Let $G \subset \hat{\mathbb{Z}}^{\ast}$ be the subgroup furnished by Theorem \ref{thm:lombardo} and let $C$ be as in Theorem \ref{thm:lombardo}. By Theorem \ref{thm:lombardo}, the index of $G$ in $\hat{\mathbb{Z}}^{\ast}$ is bounded by $C$. We deduce from Theorem \ref{thm:lombardo} together with Lemma \ref{lem:uniformhomothety} that for each $a \in G$, there is $\sigma \in \Gal(\bar{\mathbb{Q}}/\mathbb{Q})$ such that $\sigma(\pi(p)) = \pi(p)$ and
\begin{align*}
\sigma(N_1p) = \sigma(\varphi(q_r+q_t)) = \varphi(q_r+\sigma(q_t)) = \\
\varphi(q_r) + \sigma(\varphi(q_t)) = \varphi(q_r) + a^{2}\varphi(q_t).
\end{align*}

Let $N_3$ denote the order of $\varphi(q_t)$. It follows that
\[ [\mathbb{Q}(N_1p):\mathbb{Q}(\pi(p))] \geq \frac{\phi(N_3)}{2^{\omega(N_3)}(2C)}.\]
If $N_3 \geq 3$, then this implies as in the proof of Theorem \ref{thm:effectiveisogenymm} that
\[ [\mathbb{Q}(N_1p):\mathbb{Q}(\pi(p))] \geq \frac{N_3^{1-\frac{7\log 2}{5 \log \log N_3}}}{2C(2\log N_3 +1)}.\]

On the other hand, we have
\[ [\mathbb{Q}(N_1p):\mathbb{Q}(\pi(p))] \leq [\mathbb{Q}(p):\mathbb{Q}(\pi(p))] \leq 3D_2\]
by \eqref{eq:boundfordegpoverpip}. Suppose that $N_3 \geq \exp(1.9\cdot10^{10})$. It follows that
\[ 3D_2 \geq \frac{N_3^{1-\frac{7\log 2}{5 \log\left(1.9\cdot10^{10}\right)}}}{2C(2\log N_3 +1)} \geq \frac{N_3^{\frac{19}{20}}}{2C(2\log N_3 +1)}.\]

The function $x \mapsto x^{\frac{1}{20}} \cdot (2(2\log x+1))^{-1}$ is monotonically increasing for $x \geq \exp(39/2)$ and its value at $\exp(1.9\cdot10^{10})$ is larger than $1$. We deduce that
\[  \frac{N_3^{\frac{19}{20}}}{2C(2\log N_3 +1)} \geq \frac{N_3^{\frac{9}{10}}}{C},\]
which implies together with the above that
\[ N_3 \leq \max\left\{(3CD_2)^{\frac{10}{9}},\exp(1.9\cdot10^{10})\right\}.\]
Thanks to Theorem \ref{thm:lombardo}, we obtain that
\[ N_3 \leq \exp\left(2.12\cdot10^{10}\right)\max\{1,h(E_0)\}^{13773}D_2^2.\]
Since $N \leq N_1N_2$ and $N_2 \leq N_3(\deg \varphi)$, we now obtain the desired bound for $N$ from \eqref{eq:boundfornone} and \eqref{eq:boundfordegvarphi}.
\end{proof}

\section{Explicit results in the Legendre family -- varieties}\label{sec:legendrevar}

In this section, we will prove Theorem \ref{thm:effectiveisogenymmposdim}. Recall that $\mathcal{E}^{(g)}$ denotes the $g$-fold fibered power of the Legendre family, which admits a natural immersion into $\mathbb{P}^1_{\mathbb{Q}} \times_{\mathbb{Q}} \left(\mathbb{P}^2_{\mathbb{Q}}\right)^g$ ($g \in \mathbb{N}$). Unless explicitly stated otherwise, degrees of subvarieties of (base changes of) $\mathcal{E}^{(g)}$ will always be taken with respect to (base changes of) this immersion. The structural morphism of $\mathcal{E}^{(g)}$ is denoted by $\pi: \mathcal{E}^{(g)} \to Y(2) = \mathbb{A}^1_{\mathbb{Q}}\backslash\{0,1\}$. By abuse of notation, we denote the base change of $\pi$ to a number field $K$ by $\pi$ as well.

First, we prove a series of preliminary lemmata about elliptic curves and the Legendre family, concerning the degree of abelian subvarieties of powers of elliptic curves and the behaviour of the degree under homomorphisms.

\begin{lem}\label{lem:degabsub}
Let $K$ be a field with a fixed algebraic closure $\bar{K}$ and let $E \subset \mathbb{P}^2_K$ be a Weierstrass model of an elliptic curve with $\End E_{\bar{K}} \simeq \mathbb{Z}$. Let $g \in \mathbb{N}$ and let $B$ be an algebraic subgroup of $E^g \subset (\mathbb{P}_K^2)^{g}$ of codimension $k > 0$. Then there exists an effective constant $c(g)$, depending only on $g$, such that the following hold:
\begin{enumerate}
\item If $B$ is the kernel of a homomorphism from $E^g$ to $E^k$ induced by a $k \times g$-matrix $M_B$ with integer entries, then
\[\deg B = (g-k)!3^{g-k}\sum_{\Delta}{\Delta^2},\]
where the sum runs over the maximal minors of $M_B$.
\item If $B$ is irreducible, then $B$ is the kernel of a homomorphism from $E^g$ to $E^k$ induced by a $k \times g$-matrix $M_B$ with integer entries whose absolute value is bounded by $c(g)\sqrt{\deg B}$.
\end{enumerate}
\end{lem}

\begin{proof}
For (1), let $l_i$ denote the class modulo rational equivalence of the pull-back of a line in $\mathbb{P}^2_K$ to $\left(\mathbb{P}^2_K\right)^g$ under projection to the $i$-th factor ($i=1,\hdots,g$). If $[B]$ denotes the class modulo rational equivalence of $B \subset (\mathbb{P}_K^2)^g$, then $[B] \cdot l_i^{\cdot 2} = 0$ for all $i$ and so
\[ \deg B = [B] \cdot (l_1+\cdots+l_g)^{\cdot (g-k)} = (g-k)!\sum_{I \subset \{1,\hdots,g\},|I| = g-k}{[B]\cdot\prod_{i \in I}{l_i}}.\]
If $\pi_I: E^g \to E^{g-k}$ denotes the projection associated to $I \subset \{1,\hdots,g\}$, then the projection formula implies that
\[ [B]\cdot\prod_{i \in I}{l_i} = \begin{cases} 3^{g-k}\#(\ker \pi_I|_B) & \text{if $\pi_I|_B$ is finite},\\
0 & \text{otherwise.}
\end{cases}\]
The formula from the lemma now follows.

For (2), set
\[ \Lambda = \{(a_1,\hdots,a_g) \in \mathbb{Z}^g; a_1x_1+\cdots+a_gx_g = 0_E \mbox{ for all } (x_1,\hdots,x_g) \in B(\bar{K})\}.\]
Since $\End E_{\bar{K}} \simeq \mathbb{Z}$, the free abelian group $\Lambda$ has rank $k$. Since $B$ is irreducible, we have $(\mathbb{Q} \cdot \Lambda) \cap \mathbb{Z}^g = \Lambda$. Choosing a basis of $\Lambda$ yields the rows of a $k \times g$-matrix $M_B$ such that $B$ is an irreducible component of the kernel of the homomorphism from $E^g$ to $E^k$ induced by $M_B$. Since $(\mathbb{Q} \cdot \Lambda) \cap \mathbb{Z}^g = \Lambda$, the rows of $M_B$ can be completed to a basis of $\mathbb{Z}^g$. It follows that the kernel of the homomorphism induced by $M_B$ is isomorphic to $E^{g-k}$ and hence equal to $B$.

Let $D$ denote the ($k$-dimensional) volume of the parallelotope spanned in $\mathbb{R} \cdot \Lambda$ by the rows of $M_B$. Let $M_B^t$ denote the transpose of $M_B$, then $D^2$ is equal to the Gram determinant $\det(M_BM_B^t)$. The Cauchy-Binet formula then implies together with part (1) that
\[ D = \sqrt{\sum_{\Delta}{\Delta^2}} \leq \sqrt{\deg B},\]
where the sum runs over the maximal minors of $M_B$.

Let $\lambda_i$ denote the $i$-th successive minimum of $\Lambda$ in $\mathbb{R} \cdot \Lambda$ with respect to the distance function induced by the Euclidean distance on $\mathbb{R}^g$ ($i=1,\hdots,k$). We have $\lambda_i \geq 1$ for all $i$ and therefore Theorem V in Section VIII.4.3 of \cite{MR1434478} implies that $\lambda_i \leq 2^k\nu_k^{-1}D$ for all $i$, where $\nu_k$ denotes the volume of the unit ball in $\mathbb{R}^k$.

By the Corollary in Section VIII.5.2 of \cite{MR1434478}, we can then find a basis $\{v_1,\hdots,v_k\}$ of $\Lambda$ such that the Euclidean norm $\lVert v_i \rVert$ of $v_i$ satisfies
\[ \lVert v_i \rVert \leq k2^{k}\nu_k^{-1}D \quad (i=1,\hdots,k).\]
Thus, we can choose $M_B$ such that all its entries are at most $c(g)\sqrt{\deg B}$ in absolute value. The lemma follows.
\end{proof}

The next lemma will be used to control the behaviour of the degree of subvarieties of fibered powers of the Legendre family under homomorphisms.

\begin{lem}\label{lem:addpoly}
Let $K$ be a field and let $n \in \mathbb{Z}$. Then the following hold:
\begin{enumerate}
\item Let $p_i: \mathcal{E}^{(3)}_K \to \mathcal{E}_K$ denote the canonical projections ($i=1,2,3$) and let $\Gamma \subset \mathcal{E}^{(3)}_K$ be the graph of addition on $\mathcal{E}_K$ such that $(p_1+p_2)|_{\Gamma} = (p_3)|_{\Gamma}$. Then $\Gamma$ is an irreducible component of the intersection in $\mathbb{P}^1_K \times_K \left(\mathbb{P}^2_K\right)^3$ of $\mathcal{E}^{(3)}_K$ with the zero locus of a multihomogeneous polynomial of multidegree $(0,1,1,1)$ with integer coefficients.
\item Let $p_i: \mathcal{E}^{(2)}_K \to \mathcal{E}_K$ denote the canonical projections ($i=1,2$), let $[n]: \mathcal{E}_K \to \mathcal{E}_K$ denote the multiplication-by-$n$ morphism, and let $\Gamma_n \subset \mathcal{E}^{(2)}_K$ be the graph of $[n]$ such that $([n] \circ p_1)|_{\Gamma_n} = (p_2)|_{\Gamma_n}$. Then $\Gamma_n$ is an irreducible component of the intersection in $\mathbb{P}^1_K \times_K \left(\mathbb{P}^2_K\right)^2$ of $\mathcal{E}^{(2)}_K$ with the zero locus of a multihomogeneous polynomial of multidegree $(e,n^2,1)$ with rational coefficients, where $e \leq n^2$.
\end{enumerate}
\end{lem}

\begin{proof}
We fix an algebraic closure $\bar{K}$ of $K$.

For (1), if $[X_i:Y_i:Z_i]$ are projective coordinates on the $i$-th $\mathbb{P}^2_K$-factor ($i=1,2,3$), then the polynomial is equal to
\[ \begin{vmatrix} X_1 & Y_1 & Z_1 \\ X_2 & Y_2 & Z_2 \\ X_3 & -Y_3 & Z_3 \end{vmatrix} = X_1Y_2Z_3-Y_1X_2Z_3+X_1Z_2Y_3-Z_1X_2Y_3+Y_1Z_2X_3-Z_1Y_2X_3.\]

For (2), let $[\Lambda:M]$ be projective coordinates on $\mathbb{P}^1_K$ and let $[X_i:Y_i:Z_i]$ be projective coordinates on the $i$-th $\mathbb{P}^2_K$-factor ($i=1,2$). If $n = 0$, the polynomial $Z_2$ has the desired property. Hence, we can assume without loss of generality that $n \neq 0$. Since changing the sign of $Y_2$ sends $\Gamma_n$ to $\Gamma_{-n}$, we can even assume that $n > 0$.

We know that there exist polynomials $A_n, B_n \in \mathbb{Q}[X,\Lambda]$ such that for $(\lambda,[x:y:1]) \in \mathcal{E}(\bar{K}) \subset \mathbb{A}^1(\bar{K}) \times \mathbb{P}^2(\bar{K})$, we have $B_n(x,\lambda) = 0$ if and only if $x$ is the abscissa of a non-zero torsion point of $\mathcal{E}_{\lambda}$ of order dividing $n$, and
\[ [n](\lambda,[x:y:1]) = (\lambda,[A_n(x,\lambda)/B_n(x,\lambda):\ast:1])\]
if $B_n(x,\lambda) \neq 0$. Furthermore, we have $\deg_X A_n = n^2$, $\deg_X B_n = n^2-1$, and $A_n$ is monic as a polynomial in $X$. For all of this, see \cite{MR3179679}.

It follows that $\Gamma_n$ is an irreducible component of the intersection in $\mathbb{P}^1_K \times_K \left(\mathbb{P}^2_K\right)^2$ of $\mathcal{E}^{(2)}_K$ with the zero locus of
\[ Z_2\tilde{A}_n(X_1,Z_1,\Lambda,M) - X_2\tilde{B}_n(X_1,Z_1,\Lambda,M),\]
where
\[ \tilde{A}_n(X_1,Z_1,\Lambda,M) = Z_1^{d}M^{e}A_n\left(\frac{X_1}{Z_1},\frac{\Lambda}{M}\right)\]
and
\[ \tilde{B}_n(X_1,Z_1,\Lambda,M) = Z_1^{d}M^{e}B_n\left(\frac{X_1}{Z_1},\frac{\Lambda}{M}\right)\]
for $d = \max\{\deg_X A_n, \deg_X B_n\} = n^2$ and $e = \max\{\deg_\Lambda A_n, \deg_\Lambda B_n\}$.

By Lemma 2.2 in \cite{MR3179679}, we have $\Lambda^{n^2}A_n(\Lambda^{-1}X,\Lambda^{-1}) = A_n(X,\Lambda)$ and $\Lambda^{n^2-1}B_n(\Lambda^{-1}X,\Lambda^{-1}) = B_n(X,\Lambda)$. This implies that $e \leq n^2$. The lemma follows.
\end{proof}

The next lemma gives the promised control over the degree of images and pre-images of subvarieties under homomorphisms between fibered powers of the Legendre family.

\begin{lem}\label{lem:degim}
Let $K$ be a field and let $g, k \in \mathbb{N}$. Let $\psi: \mathcal{E}^{(g)}_K \to \mathcal{E}^{(k)}_K$ be a homomorphism of abelian schemes, defined by a $k \times g$-matrix
\[ M = \begin{pmatrix} a_{1,1} & \cdots & a_{1,g} \\ \vdots & & \vdots \\ a_{k,1} & \cdots & a_{k,g}\end{pmatrix}\]
with integer coefficients.
 
Set
\[ \Pi(M) = \prod_{i=1}^{k}\prod_{j=1}^{g}{\max\{1,|a_{i,j}|\}}.\]
There exists an effective constant $C(k,g)$, depending only on $k$ and $g$, such that the following hold:
\begin{enumerate}
\item Let $\mathcal{V} \subset \mathcal{E}^{(g)}_K$ be a subvariety. Then
\[ \deg \psi(\mathcal{V}) \leq C(k,g)\Pi(M)^2(\deg \mathcal{V}).\]
\item Let $\mathcal{W} \subset \mathcal{E}^{(k)}_K$ be a subvariety. Then
\[ \deg \psi^{-1}(\mathcal{W}) \leq C(k,g)\Pi(M)^2(\deg \mathcal{W}).\]
\end{enumerate}
\end{lem}

\begin{proof}
Let $\Gamma_{\psi} \subset \mathcal{E}^{(g+k)}_K \subset \mathbb{P}^1_K \times_K \left(\mathbb{P}^2_K\right)^{g+k}$ denote the graph of $\psi$. We first estimate $\deg \Gamma_{\psi}$. Constants $c_1, c_2, \hdots$ will be effective and depend only on $k$ and $g$.

For $m \in \mathbb{N}$ and an $m$-tuple $I$ of distinct elements of $\{1,\hdots,g+k(2g-1)\}$, let $p_I: \mathcal{E}_K^{(g+k(2g-1))} \to \mathcal{E}_K^{(m)}$ denote the corresponding projection, where the order of the factors is given by the order of $I$. As in Lemma \ref{lem:addpoly}, we write $\Gamma$ for the graph of addition on $\mathcal{E}_K$ and $\Gamma_n$ for the graph of multiplication by $n$ on $\mathcal{E}_K$ ($n \in \mathbb{Z}$).

If $g \geq 3$, set
\begin{multline}\label{eq:biguglyequation}
\tilde{\Gamma}_{\psi} = \bigcap_{i=1}^{k}\bigcap_{j=1}^{g}{p_{(j,k+gi+j)}^{-1}\left(\Gamma_{a_{i,j}}\right)} \cap \bigcap_{i=1}^{k}{p_{(k+gi+1,k+gi+2,k+kg+(g-2)i+3)}^{-1}\left(\Gamma\right)}\\
\cap  \bigcap_{i=1}^{k}\bigcap_{j=3}^{g-1}{p_{(k+gi+j,k+kg+(g-2)i+j,k+kg+(g-2)i+j+1)}^{-1}\left(\Gamma\right)} \cap \bigcap_{i=1}^{k}{p_{(k+g(i+1),k+kg+(g-2)i+g,g+i)}^{-1}\left(\Gamma\right)},
\end{multline}
then we have
\begin{equation}\label{eq:smallniceequation}
\Gamma_{\psi} = p_{(1,\hdots,g+k)}(\tilde{\Gamma}_{\psi}).
\end{equation}
If $g \in \{1,2\}$, then the same equality holds, but the definition of $\tilde{\Gamma}_{\psi}$ has to be slightly modified.

The degree of the zero locus in $\mathbb{P}^1_K \times_K \left(\mathbb{P}^2_K\right)^{g+k(2g-1)}$ of a multihomogeneous polynomial of multidegree $\left(d_0,d_1,\hdots,d_{g+k(2g-1)}\right)$ is bounded from above by
\[ c_1\left(d_0+d_1+\cdots+d_{g+k(2g-1)}\right).\]

It then follows from Theorem \ref{thm:bezout} and Lemma \ref{lem:addpoly} that the degree of the pre-image of $\Gamma_{a_{i,j}}$ under some projection is bounded by $c_2\max\{1,a_{i,j}^2\}$ whereas the degree of the pre-image of $\Gamma$ under some projection is bounded by $c_3$. Together with \eqref{eq:biguglyequation} and Theorem \ref{thm:bezout}, this implies that $\deg \tilde{\Gamma}_{\psi} \leq c_4\Pi(M)^2$. We deduce from \eqref{eq:smallniceequation} together with Lemma \ref{lem:degproj} that also $\deg \Gamma_{\psi} \leq c_4\Pi(M)^2$.

We can assume without loss of generality that $\mathcal{V}$ and $\mathcal{W}$ are irreducible. We consider $\mathcal{V} \times_K \left(\mathbb{P}^2_K\right)^{k} \subset \mathbb{P}^1_K \times_K \left(\mathbb{P}^2_K\right)^{g+k}$ and $\mathcal{W} \times_K \left(\mathbb{P}^2_K\right)^{g} \subset \mathbb{P}^1_K \times_K \left(\mathbb{P}^2_K\right)^{g+k}$ (with the appropriate ordering of the factors). These varieties have degrees
\[ \binom{\dim \mathcal{V}+2k}{\dim \mathcal{V},2,2,\hdots,2}(\deg \mathcal{V}) \leq c_5 \deg \mathcal{V}\]
and
\[ \binom{\dim \mathcal{W}+2g}{\dim \mathcal{W},2,2,\hdots,2}(\deg \mathcal{W}) \leq c_6 \deg \mathcal{W}\]
respectively. To obtain $\psi(\mathcal{V})$ and $\psi^{-1}(\mathcal{W})$ respectively, we intersect those varieties with $\Gamma_{\psi}$ and project to certain factors of the product. The lemma then follows from Theorem \ref{thm:bezout} and Lemma \ref{lem:degproj}.
\end{proof}

In the next lemma, we apply Lemma \ref{lem:degim} to the special case of the addition morphism.

\begin{lem}\label{lem:degadd}
Let $K$ be a field and let $g \in \mathbb{N}$. There exists an effective constant $C(g)$, depending only on $g$, such that the following holds: Let $\mathcal{V}, \mathcal{W} \subset \mathcal{E}^{(g)}_K$ be two subvarieties. Then
\[ \deg(\mathcal{V}+\mathcal{W}) \leq C(g)(\deg \mathcal{V})(\deg \mathcal{W}).\]
\end{lem}

\begin{proof}
Let $\psi: \mathcal{E}^{(2g)}_K \to \mathcal{E}^{(g)}_K$ denote the addition morphism of $\mathcal{E}^{(g)}_K$. We have $\mathcal{V}+\mathcal{W} = \psi(\mathcal{V} \times_{Y(2)_K} \mathcal{W})$, where we consider $\mathcal{V} \times_{Y(2)_K} \mathcal{W} \subset \mathcal{E}^{(2g)}_K$ with its reduced subscheme structure. We can assume without loss of generality that $\mathcal{V}$ and $\mathcal{W}$ are irreducible. Now
\[ \mathcal{V} \times_{Y(2)_K} \mathcal{W} \hookrightarrow \mathcal{V} \times_K \mathcal{W} \subset \mathcal{E}^{(g)}_K \times_K \mathcal{E}^{(g)}_K \subset \left(\mathbb{P}^{1}_K \times_K (\mathbb{P}^2_K)^g\right)^2.\]

The degree of $\mathcal{V} \times_{Y(2)_K} \mathcal{W}$ as a subvariety of $\mathcal{E}^{(2g)}_K \subset \mathbb{P}^{1}_K \times_K (\mathbb{P}^2_K)^{2g}$ is bounded from above by its degree as a subvariety of $\mathcal{V} \times_K \mathcal{W} \subset \left(\mathbb{P}^{1}_K \times_K (\mathbb{P}^2_K)^g\right)^2$ by Lemma \ref{lem:degproj}.

Let $\Delta$ denote the pre-image in $\left(\mathbb{P}^{1}_K \times_K (\mathbb{P}^2_K)^g\right)^2$ of the diagonal in $\left(\mathbb{P}^1_K\right)^2$, then $\mathcal{V} \times_{Y(2)_K} \mathcal{W}$ is equal to the intersection $(\mathcal{V} \times_{K} \mathcal{W}) \cap \Delta$ inside $\left(\mathbb{P}^{1}_K \times_K (\mathbb{P}^2_K)^g\right)^2$. It follows from Theorem \ref{thm:bezout} that the degree of $\mathcal{V} \times_{Y(2)_K} \mathcal{W}$ with respect to its immersion into $\left(\mathbb{P}^{1}_K \times_K (\mathbb{P}^2_K)^g\right)^2$ is bounded from above by
\[\deg(\mathcal{V} \times_K \mathcal{W}) \cdot \deg(\Delta) = \binom{\dim \mathcal{V}+\dim\mathcal{W}}{\dim \mathcal{V}}(\deg \mathcal{V})(\deg \mathcal{W})(\deg \Delta).\]

The lemma now follows from $\mathcal{V}+\mathcal{W} = \psi(\mathcal{V} \times_{Y(2)_K} \mathcal{W})$ and Lemma \ref{lem:degim}.
\end{proof}

We can now prove Theorem \ref{thm:effectiveisogenymmposdim}.

\begin{proof}[Proof of Theorem \ref{thm:effectiveisogenymmposdim}]
Throughout the proof, we use constants $c_1, c_2, \hdots$ that are effective and depend only on $g$. We induct on $(g,\dim \mathcal{V})$ with respect to the lexicographic order. We can assume without loss of generality that $\mathcal{V}$ is irreducible (although not necessarily geometrically irreducible). We proceed by distinguishing various cases.

\emph{Case 1: $\dim \pi(\mathcal{V}) = 0$.}

Let $p \in \mathcal{V}(\bar{K})$ be torsion on $\mathcal{E}^g_{\pi(p)}$ such that $\mathcal{E}_{\pi(p)}$ is isogenous to $(E_0)_{\bar{K}}$. Since $\mathcal{V}$ is defined over $K$ and irreducible, we have
\[ [K(\pi(p)):K] = |\pi(\mathcal{V}(\bar{K}))| \leq \deg \mathcal{V}.\]
We now deduce from Th\'{e}or\`{e}me 1.4 in \cite{MR3225452} that conclusion (2) of the theorem is satisfied for a suitable choice of $\gamma(g)$.

\emph{Case 2: $\pi(\mathcal{V}) = Y(2)_K$.}

By Theorem \ref{thm:faltings}, the Zariski closure of $\mathcal{V}$ in $\mathbb{P}^1_K \times_K (\mathbb{P}^2_K)^g$ is defined by multihomogeneous polynomials of multidegree at most $(\deg \mathcal{V},\hdots,\deg \mathcal{V})$. Hence, $\mathcal{V}$ is defined in $\mathcal{E}^{(g)}_K \subset \mathbb{P}^1_K \times_K (\mathbb{P}^2_K)^g$ by multihomogeneous polynomials of degree at most $\deg \mathcal{V}$ in each set of projective coordinates. The same then holds for $\mathcal{V}_\xi$ in $\mathcal{E}^g_\xi \subset \left(\mathbb{P}^2_{\overline{K(Y(2))}}\right)^g$. It follows from Theorem \ref{thm:bezoutii} that
\begin{equation}\label{eq:degboundgenfiber}
\deg \mathcal{V}_\xi \leq c_1(\deg \mathcal{V})^{g-\dim \mathcal{V}_\xi} = c_1(\deg \mathcal{V})^{g+1-\dim \mathcal{V}}.
\end{equation}

\emph{Case 2.1: The stabilizer $\Stab(\mathcal{V}_\xi,\mathcal{E}^g_{\xi})$ of $\mathcal{V}_\xi$ is positive-dimensional.}

Let $k < g$ be the codimension of $\Stab(\mathcal{V}_\xi,\mathcal{E}^g_{\xi})$ and let $A$ be the identity component of $\Stab(\mathcal{V}_\xi,\mathcal{E}^g_{\xi})$. We can assume without loss of generality that $k > 0$. Starting with an irreducible component of $\mathcal{V}_\xi - x_0$ that contains $A$ (for an arbitrary $x_0 \in \mathcal{V}_\xi(\overline{K(Y(2))})$), we successively intersect with $\mathcal{V}_\xi - x$ for some $x \in \mathcal{V}_\xi(\overline{K(Y(2))})$ and take an irreducible component of the intersection that contains $A$. After doing this at most $\dim \mathcal{V}_\xi - \dim A$ times and choosing $x$ sufficiently general in each step, we obtain $A$ itself. It follows from Theorem \ref{thm:bezout} and \eqref{eq:degboundgenfiber} that
\begin{equation}\label{eq:bounddegstab}
\deg A \leq (\deg \mathcal{V}_\xi)^{\dim \mathcal{V}_\xi - \dim A+1} \leq c_2(\deg \mathcal{V})^{(g+1-\dim \mathcal{V})\dim \mathcal{V}_\xi}.
\end{equation}

By Lemma \ref{lem:degabsub}(2), the identity component $A$ is equal to the kernel of a homomorphism from $\mathcal{E}^g_{\xi}$ to $\mathcal{E}^k_{\xi}$ induced by a $k \times g$ matrix $M_A$ with integer entries of absolute value at most $c_3\sqrt{\deg A}$. Now $M_A$ induces a homomorphism $\psi_A: \mathcal{E}^{(g)}_K \to \mathcal{E}^{(k)}_K$. By abuse of notation, we also write $\psi_A$ for the induced homomorphism from $\mathcal{E}^g_{\xi}$ to $\mathcal{E}^k_{\xi}$. We write $\pi'$ for the structural morphism $\mathcal{E}^{(k)} \to Y(2)$.

Set $\mathcal{V}' = \psi_A(\mathcal{V})$. We have $\mathcal{V} = \psi_A^{-1}(\mathcal{V}')$ by construction. By Lemma \ref{lem:degim}, the above bound for the absolute value of the entries of $M_A$, and \eqref{eq:bounddegstab}, we have
\begin{equation}\label{eq:bounddegvprime}
\deg \mathcal{V}' \leq c_4(\deg A)^{kg}(\deg \mathcal{V}) \leq c_5(\deg \mathcal{V})^{kg(g+1-\dim \mathcal{V})\dim \mathcal{V}_\xi+1}.
\end{equation}

Let now $p \in \mathcal{V}(\bar{K})$ be torsion on $\mathcal{E}^g_{\pi(p)}$ such that $\mathcal{E}_{\pi(p)}$ is isogenous to $(E_0)_{\bar{K}}$. Then $p' := \psi_A(p)\in \mathcal{V}'(\bar{K})$ is torsion on $\mathcal{E}^k_{\pi'(p')}$ and $\mathcal{E}_{\pi'(p')}$ is isogenous to $(E_0)_{\bar{K}}$. We apply induction on $(g,\dim \mathcal{V})$ and use that the theorem holds for $p'$, $\mathcal{V}'$, and $\mathcal{E}^{(k)}_K$.

If conclusion (2) of the theorem holds for $p'$, then we are done thanks to $\pi'(p') = \pi(p)$ and \eqref{eq:bounddegvprime}.

If conclusion (1) of the theorem holds for $p'$, we find that there exists a torsion point $q' \in \mathcal{E}^k_{\xi}$ and an abelian subvariety $B'$ of $\mathcal{E}^k_{\xi}$ such that $p' \in \overline{q'+B'}(\bar{K})$ and $\overline{q'+B'} \subset \mathcal{V}'$, where $\overline{q'+B'}$ denotes the Zariski closure in $\mathcal{E}^{(k)}_K$ of the image of $q'+B'$ under the natural morphism $\mathcal{E}^k_{\xi} \to \mathcal{E}^{(k)}_K$. Furthermore, the order of $q'$ and the degree of $B'$ are bounded as in the theorem in terms of $k$ and $\deg \mathcal{V}'$.

Set $B = \psi_A^{-1}(B')$, then $B$ is an abelian subvariety of $\mathcal{E}^g_{\xi}$. There exists a torsion point $q \in \mathcal{E}^g_{\xi}$ of order dividing the order of $q'$ such that $q+B = \psi_A^{-1}(q'+B')$. Let $\overline{q+B}$ denote the Zariski closure in $\mathcal{E}^{(g)}_K$ of the image of $q+B$ under the natural morphism $\mathcal{E}^g_{\xi} \to \mathcal{E}^{(g)}_K$. We have $\overline{q+B} \subset \psi_A^{-1}(\overline{q'+B'})$, but actually equality holds since both varieties are irreducible of the same dimension. Since $p' = \psi_A(p)$ and $\mathcal{V} = \psi_A^{-1}(\mathcal{V}')$, it follows that $p \in \overline{q+B}(\bar{K})$ and $\overline{q+B} \subset \mathcal{V}$.

We are now again done thanks to \eqref{eq:bounddegvprime}, provided that we can bound the degree of $B$ in the required way. Let $k'$ denote the codimension of $B'$. We have $k' > 0$ since $k > 0$. By Lemma \ref{lem:degabsub}(2), $B'$ is the kernel of a homomorphism from $\mathcal{E}^k_{\xi}$ to $\mathcal{E}^{k'}_{\xi}$ induced by a $k' \times k$-matrix with integer entries of absolute value at most $c_6\sqrt{\deg B'}$. Together with the above bound for the absolute value of the entries of $M_A$, this implies that $B = \psi_A^{-1}(B')$ is the kernel of a homomorphism from $\mathcal{E}^g_{\xi}$ to $\mathcal{E}^{k'}_{\xi}$ induced by a $k' \times g$-matrix with integer entries of absolute value at most $c_7\sqrt{\deg A \deg B'}$. We then deduce from Lemma \ref{lem:degabsub}(1) and Hadamard's determinant inequality that
\[\deg B \leq c_8((\deg A)(\deg B'))^{k'}.\]
Together with \eqref{eq:bounddegstab}, \eqref{eq:bounddegvprime}, and the bound for $\deg B'$ furnished by the inductive hypothesis, this completes the proof of the theorem in Case 2.1.

\emph{Case 2.2: The stabilizer of $\mathcal{V}_\xi$ is finite.}

Let $p \in \mathcal{V}(\bar{K})$ be torsion on $\mathcal{E}^g_{\pi(p)}$ such that $\mathcal{E}_{\pi(p)}$ is isogenous to $(E_0)_{\bar{K}}$.

\emph{Case 2.2.1: The point $p$ lies in a translate of a positive-dimensional abelian subvariety of $\mathcal{E}^g_{\pi(p)}$ that is contained in $\mathcal{V}$.}

Let $Z$ be an irreducible component of $\mathcal{V}_{\xi}$. As every abelian subvariety of $\mathcal{E}^g_{\xi}$ is defined over $K(Y(2))$, the stabilizer of $Z$ is finite as well. Hence, there exist points $x_1, \hdots, x_{\dim Z+1} \in Z(\overline{K(Y(2))})$ such that 
\[ (Z-x_1) \cap \hdots \cap (Z-x_{\dim Z +1})\]
is a finite set. By Theorem \ref{thm:bezout} and \eqref{eq:degboundgenfiber}, the number of $t$ such that $t+Z = Z$ is then bounded by $(\deg Z)^{\dim Z+1} \leq c_9(\deg \mathcal{V})^{(g+1-\dim \mathcal{V})\dim \mathcal{V}}$.

Let $\mathcal{K}$ denote the set of pairs $(Z_1,Z_2)$ of irreducible components of $\mathcal{V}_{\xi}$ such that there exists $t \in (\mathcal{E}^g_\xi)_{\tors}$ with $Z_1 = t + Z_2$. For example, $\mathcal{K}$ contains all pairs $(Z,Z)$, where $Z$ is an irreducible component of $\mathcal{V}_{\xi}$. The cardinality of $\mathcal{K}$ is at most equal to $(\deg \mathcal{V}_\xi)^2$, which is at most equal to $c_{10}(\deg \mathcal{V})^{2(g+1-\dim \mathcal{V})}$ by \eqref{eq:degboundgenfiber}. If $(Z_1,Z_2) \in \mathcal{K}$, let $t(Z_1,Z_2)$ be an arbitrary torsion point such that $Z_1 = t(Z_1,Z_2) + Z_2$. The set of torsion points $t$ such that $Z_1 = t + Z_2$ for some irreducible components $Z_1, Z_2$ of $\mathcal{V}_{\xi}$ is then equal to
\[ \{ t(Z_1,Z_2) + t'; (Z_1,Z_2) \in \mathcal{K}, t' \in \Stab(Z_2,\mathcal{E}^g_{\xi})\}.\]

The cardinality of this set is at most equal to $c_{11}(\deg \mathcal{V})^{(\dim \mathcal{V}+2)(g+1-\dim \mathcal{V})}$, so there exists a natural number $N \leq c_{12}(\deg \mathcal{V})^{(\dim \mathcal{V}+2)(g+1-\dim \mathcal{V})}$ with the property that no torsion point $t$ of order $N$ satisfies $Z_1 = t + Z_2$ for some irreducible components $Z_1, Z_2$ of $\mathcal{V}_{\xi}$. Let $\mathcal{T} \subset \mathcal{E}^{(g)}_K$ be the Zariski closure of the image of the set of all torsion points of order $N$ of $\mathcal{E}^g_\xi$ under the natural morphism $\mathcal{E}^g_\xi \to \mathcal{E}^{(g)}_K$. Then $\mathcal{T}$ contains the torsion points of order $N$ of all fibers of $\mathcal{E}^{(g)}_K$ (cf. the proof of Lemma \ref{lem:not-zariskidense-too}).

Since every positive-dimensional abelian variety contains a torsion point of order $N$, we find that $p \in (\mathcal{V} \cap (\mathcal{T}+\mathcal{V}))(\bar{K})$. At the same time, looking at the geometric generic fiber, we see that $\mathcal{V} \cap (\mathcal{T}+\mathcal{V}) \subsetneq \mathcal{V}$ by our choice of $N$. If $[N]: \mathcal{E}^{(g)}_K \to \mathcal{E}^{(g)}_K$ denotes the multiplication-by-$N$ morphism and $\epsilon: Y(2)_K \to \mathcal{E}^{(g)}_K$ denotes the zero section, then $\mathcal{T}$ is a union of irreducible components of $[N]^{-1}(\epsilon(Y(2)_K))$. Hence Lemma \ref{lem:degim}(2) implies that
\[ \deg \mathcal{T} \leq \deg [N]^{-1}(\epsilon(Y(2)_K)) \leq c_{13}N^{2g}.\]

By Lemma \ref{lem:degadd}, we can then estimate
\[ \deg(\mathcal{T}+\mathcal{V}) \leq c_{14}N^{2g}(\deg \mathcal{V}).\]
It follows from this together with Theorem \ref{thm:bezout} that
\[ \deg(\mathcal{V} \cap (\mathcal{T}+\mathcal{V})) \leq (\deg \mathcal{V})\deg(\mathcal{T}+\mathcal{V}) \leq c_{14}N^{2g}(\deg \mathcal{V})^2.\]
We are now again done by the above bound on $N$ and induction on $(g,\dim \mathcal{V})$.

\emph{Case 2.2.2: The point $p$ does not lie in any translate of a positive-dimensional abelian subvariety of $\mathcal{E}^g_{\pi(p)}$ that is contained in $\mathcal{V}$.}

In this case, we can apply Proposition \ref{prop:hindry} to bound the order $N_p$ of $p$. By Lemma \ref{lem:uniformhomothety}, there exists for any $a \in \hat{\mathbb{Z}}^{\ast}$ some $\sigma_a \in \Gal(\bar{K}/K)$ that fixes $\pi(p)$ and acts on the torsion of $\mathcal{E}_{\pi(p)}$ as multiplication by $a^{2c(E_0/K)}$, where $c(E_0/K)$ is the Serre constant as defined in Lemma \ref{lem:uniformhomothety}. Proposition \ref{prop:hindry}, applied to $p$ inside $\mathcal{V}_{\pi(p)} \subset \mathcal{E}^g_{\pi(p)} \hookrightarrow \mathbb{P}^{3^g-1}_{\bar{\mathbb{Q}}}$ over the field of definition $K(\pi(p))$, then implies together with the bound for $c(E_0/K)$ in Theorem \ref{thm:lombardo} that $N_p$ can be bounded by
\begin{equation}\label{eq:boundfornp}
\max\left\{\exp\left(\max\{2,h(E_0),[K:\mathbb{Q}]\}^{c_{15}}\right),\left(\deg \mathcal{V}_{\pi(p)}\right)^{c_{15}}\right\},
\end{equation}
where the dependency on $h(E_0)$ can be dropped if $E_0$ has CM.

We get the bound
\begin{equation}\label{eq:degboundspecfiber}
\deg \mathcal{V}_{\pi(p)} \leq c_{16}(\deg \mathcal{V})^{g+1-\dim \mathcal{V}}
\end{equation}
for $\deg \mathcal{V}_{\pi(p)}$ in the same way as the bound \eqref{eq:degboundgenfiber} for $\deg \mathcal{V}_\xi$.

\emph{Case 2.2.2.1: There exists $q \in \left(\mathcal{E}^g_{\xi}\right)_{\tors}$ such that $p \in \overline{q} \subset \mathcal{V}$, where $\overline{q}$ denotes the Zariski closure of the image of $q$ in $\mathcal{E}^{(g)}_K$ under the natural morphism $\mathcal{E}^g_{\xi} \to \mathcal{E}^{(g)}_K$.}

As $p$ does not lie in any translate of a positive-dimensional abelian subvariety of $\mathcal{E}^g_{\pi(p)}$ that is contained in $\mathcal{V}$, we know that $q$ and all of its Galois conjugates over $\bar{K}(Y(2))$ are maximal torsion cosets in $ \mathcal{V}_{\xi}$. The order of $q$ is $N_p$. It follows from Corollary 2 on p. 69 of \cite{MR890960} that there are at least $cN_p$ Galois conjugates of $q$ over $\bar{K}(Y(2))$ for an effective absolute constant $c > 0$ (note that $\mathcal{E}_{\xi}$ is isomorphic to the base change of an elliptic curve defined over $\bar{K}(j(\mathcal{E}_{\xi}))$ and that the isomorphism is defined over a field extension of $\bar{K}(Y(2))$ of degree at most 12 by Th\'eor\`eme 1.2 in \cite{R17}). But by Th\'eor\`eme 1.13 in \cite{DavidPhilippon2007}, the number of maximal torsion cosets contained in $\mathcal{V}_\xi$ is bounded from above by $\max\{2,\deg \mathcal{V}_\xi\}^{c_{17}}$. Thanks to \eqref{eq:degboundgenfiber}, the order of $q$ is then bounded as in conclusion (1) of the theorem and we are done.

\emph{Case 2.2.2.2: There exists no $q \in \left(\mathcal{E}^g_{\xi}\right)_{\tors}$ as in Case 2.2.2.1.}

The singleton $\{p\}$ is then an irreducible component of $[N_p]^{-1}(\epsilon(Y(2)_K)) \cap \mathcal{V}$, where $[N_p]: \mathcal{E}^{(g)}_K \to \mathcal{E}^{(g)}_K$ denotes the multiplication-by-$N_p$ morphism. It follows that
\[ [K(p):K] \leq \deg([N_p]^{-1}(\epsilon(Y(2)_K)) \cap  \mathcal{V}).\]
Together with Theorem \ref{thm:bezout} and Lemma \ref{lem:degim}(2), this implies that
\[ [K(\pi(p)):K] \leq [K(p):K] \leq c_{18}N_p^{2g}(\deg  \mathcal{V}).\]
The existence of an isogeny between $(E_0)_{\bar{K}}$ and $\mathcal{E}_{\pi(p)}$ of degree bounded as in conclusion (2) of the theorem now follows from this inequality together with \eqref{eq:boundfornp}, \eqref{eq:degboundspecfiber}, and Th\'{e}or\`{e}me 1.4 in \cite{MR3225452}.
\end{proof}

\section{Uniform bounds on the number of maximal torsion cosets}\label{sec:galmar}

In this section, we prove Theorem \ref{thm:countingacrossisogenyclass}. Its proof hinges on the work \cite{GT17} of Galateau-Mart\'inez.

\begin{proof}[Proof of Theorem \ref{thm:countingacrossisogenyclass}]
Theorem \ref{thm:countingacrossisogenyclass} will follow from Theorem 4.5 in \cite{GT17} once we have shown that the constant $c$ used there can be bounded in terms of only $A_0$ and $K$. Here, $c \in \mathbb{N}$ satisfies: There is some number field $L \subset \bar{K}$ over which $A$ and its embedding into $\mathbb{P}^N_{\bar{K}}$ can be defined (up to $\bar{K}$-isomorphism) such that for all $a, N \in \mathbb{N}$ with $\gcd(a,N) = 1$, there exists $\sigma \in \Gal(\bar{K}/L)$ such that $\sigma$ acts on the $N$-torsion of $A$ as multiplication by $a^c$.

If we forget for the moment the projective embedding, then the existence of such a constant $c = c(A_0,K)$ for $A_0$ (with $L = K$) is guaranteed by a theorem of Serre (Th\'eor\`eme 3 in \cite{MR1944805}, see also \cite{MR1730973}, No. 136, Th\'{e}or\`{e}me 2'). We will show that the same constant $c$ works not only for $A$, but for any quotient $B$ of $(A_0)_{\bar{K}}$ by an algebraic subgroup (that could be of positive dimension) and therefore also for quotients of these quotients etc. Furthermore, the number field $L$ can be chosen so that not only $B$ can be defined over it, but also the homomorphism $(A_0)_{\bar{K}} \to B$ (and the same for quotients of $B$ etc. to any finite ``depth"). In fact, it seems that this strengthening is already used implicitly in \cite{GT17} when passing from $A$ to $A/\Stab(V)$.

Let $\psi: (A_0)_{\bar{K}} \to B$ be a surjective homomorphism. Let $L \subset \bar{K}$ be the smallest field containing $K$ over which the algebraic subgroup $\ker \psi$ of $(A_0)_{\bar{K}}$ is defined. Then $B$ is isomorphic to $B'_{\bar{K}}$, where $B'$ is an abelian variety defined over $L$, and there exists a surjective homomorphism $\chi: (A_0)_L \to B'$ such that $(\ker \chi)_{\bar{K}} = \ker \psi$.

Suppose that $\sigma \in \Gal(\bar{K}/K)$ acts as multiplication by $a \in \hat{\mathbb{Z}}^{\ast}$ on the torsion of $(A_0)_{\bar{K}}$. For every torsion point $t \in (\ker \psi)(\bar{K})$, we therefore have $\sigma(t) = at \in (\ker \psi)(\bar{K})$. As the torsion points in $\ker \psi$ lie dense in $\ker \psi$, we deduce that $\sigma(\ker \psi) \subset \ker \psi$ and hence $\sigma(\ker \psi) = \ker \psi$. It follows that $\sigma \in \Gal(\bar{K}/L)$.

If we identify $B$ with $B'_{\bar{K}}$, then $\psi: (A_0)_{\bar{K}} \to B$ is the base change of $\chi: (A_0)_L \to B'$ and $\sigma$ fixes $\psi$. Since $\sigma$ acts as multiplication by $a$ on the torsion of $(A_0)_{\bar{K}}$, this implies that $\sigma$ also acts as multiplication by $a$ on the torsion of $B$. It is clear that this can be iterated now for quotients of $B$ etc.

We still have to take care of the projective embedding $B \hookrightarrow \mathbb{P}^N_{\bar{K}}$ that we had momentarily forgotten; for this, we might have to replace $\sigma$ by some fixed iterate, depending only on $g$. The projective embedding is associated to a symmetric very ample line bundle $\mathcal{L}$ on $B$. Up to an isomorphism of $\mathbb{P}^N_{\bar{K}}$, the embedding can be defined over any field of definition of $\mathcal{L}$ since it is projectively, so in particular linearly normal. Let $\lambda: B \to \hat{B}$ be the homomorphism induced by $\mathcal{L}$. By Th\'eor\`eme 1.2 in \cite{R17}, it is defined over a field extension $L'$ of $L$ of degree bounded in terms of $g$.

Let $\mathcal{P}$ denote the Poincar\'e line bundle on $B \times_{\bar{K}} \hat{B}$. The line bundle $\mathcal{M} = (\id_{B},\lambda)^{\ast}\mathcal{P}$ is symmetric by Theorem 8.8.4 in \cite{MR2216774} and defined over $L'$. By Proposition 6.10 in \cite{MR1304906}, $\mathcal{L}^{\otimes 2}$ and $\mathcal{M}$ are algebraically equivalent. Since both $\mathcal{L}$ and $\mathcal{M}$ are symmetric, $\mathcal{L}^{\otimes 2} \otimes \mathcal{M}^{\otimes(-1)}$ is both symmetric and antisymmetric by Theorem 8.8.3 in \cite{MR2216774}. Therefore, $\mathcal{L}^{\otimes 4} \otimes \mathcal{M}^{\otimes(-2)}$ is trivial. It follows that for any conjugate $\mathcal{L}'$ of $\mathcal{L}$ over $L'$, $\mathcal{L}^{\otimes 4} \otimes \mathcal{L}'^{\otimes (-4)}$ is trivial. This implies that $\mathcal{L} \otimes \mathcal{L}'^{\otimes(-1)}$ corresponds to a torsion point of $\hat{B}$ of order dividing $4$ and hence there are at most $4^{2g}$ possibilities for $\mathcal{L}'$ (up to isomorphism). Since the relative Picard functor $\Pic_{B'/L}$ is representable by a scheme that is locally of finite type over $L$ thanks to Th\'eor\`eme 3.1 and the following paragraph in \cite{FGAV}, this implies that $\mathcal{L}$ is defined over a field extension of $L'$ of degree bounded in terms of $g$ and we are done.

As explained at the beginning of the proof, Theorem \ref{thm:countingacrossisogenyclass} now follows from Theorem 4.5 in \cite{GT17}.
\end{proof}

\section*{Acknowledgements}

This article has grown out of a section of my PhD thesis. I thank my PhD advisor Philipp Habegger for his constant support and for many helpful and interesting discussions. I thank Philipp Habegger and Ga\"el R\'emond for helpful comments on the thesis. I thank Fabrizio Barroero and Ga\"el R\'emond for useful remarks on a preliminary version of this article. I thank Francesco Campagna, Davide Lombardo, C\'esar Mart\'inez, David Masser, Jonathan Pila, and Harry Schmidt for useful conversations and correspondence. I thank the referee for their helpful suggestions and in particular for improving the exponent in Lemma \ref{lem:uniformhomothety} from $1663200$ to $2$, which also improved the upper bounds in Theorems \ref{thm:effectiveisogenymm} and \ref{thm:effectiveisogenyml}. When I had the initial idea for this article, I was supported by the Swiss National Science Foundation as part of the project ``Diophantine Problems, o-Minimality, and Heights", no. 200021\_165525. I completed it while supported by the Early Postdoc.Mobility grant no. P2BSP2\_195703 of the Swiss National Science Foundation. I thank the Mathematical Institute of the University of Oxford and my host there, Jonathan Pila, for hosting me as a visitor for the duration of this grant.

\bibliographystyle{acm}
\bibliography{Bibliography}
\end{document}